\documentclass[11pt]{article}

\usepackage{amsfonts}
\usepackage{shuffle,yfonts,amsmath,amssymb,amsthm}
\usepackage{hyperref}
\hypersetup{pdfstartview={XYZ null null 1.00}}
\usepackage{pdfpages}

\newcommand\nc{\newcommand}
\nc{\ES}{\mathsf{ES}}
\nc{\MZV}{\mathsf{MZV}}
\nc{\CMZV}{\mathsf{CMZV}}
\nc{\MtV}{\mathsf{MtV}}
\nc{\AMtV}{\mathsf{AMtV}}
\nc{\MTV}{\mathsf{MTV}}
\nc{\MSV}{\mathsf{MSV}}
\nc{\MMV}{\mathsf{MMV}}
\nc{\MMVo}{\mathsf{MMVo}}
\nc{\MMVe}{\mathsf{MMVe}}
\nc{\AMMV}{\mathsf{AMMV}}

\nc{\sha}{\shuffle}
\nc{\si}{\sigma}
\nc{\gd}{\delta}
\nc{\ola}{\overleftarrow}
\nc{\ora}{\overrightarrow}
\nc{\lra}{\longrightarrow}
\nc{\Lra}{\Longrightarrow}
\nc\Res{{\rm Res}}
\nc\setX{{\mathsf{X}}}
\nc\fA{{\mathfrak{A}}}
\nc\evaM{{\texttt{M}}}
\nc\evaML{{\text{\em{\texttt{M}}}}}
\nc\z{{\texttt{z}}}
\nc\ta{{\texttt{a}}}
\nc\tb{{\texttt{b}}}
\nc\tc{{\texttt{c}}}
\nc\te{{\texttt{e}}}
\nc\ty{{\texttt{y}}}
\nc\tx{{\texttt{x}}}
\nc\td{{\texttt{d}}}
\nc\tz{{\texttt{z}}}
\nc\txp{{\tx_1}} 
\nc\txn{{\tx_{-1}}} 
\nc\neo{{1}}
\nc{\yi}{{1}}
\nc\one{{-1}}
\nc\om{{\omega}}
\nc\omn{\omega_{-1}}
\nc\omz{\omega_0}
\nc\omp{\omega_{1}}
\nc\eps{{\varepsilon}}
\nc{\bfp}{{\bf p}}
\nc{\bfq}{{\bf q}}
\nc{\bfu}{{\bf u}}
\nc{\bfv}{{\bf v}}
\nc{\bfw}{{\bf w}}
\nc{\bfy}{{\bf y}}
\nc{\bb}{{b}}
\nc{\bfga}{{\boldsymbol{\sl{\alpha}}}}
\nc{\bfe}{{\boldsymbol{\sl{e}}}}
\nc{\bfi}{{\boldsymbol{\sl{i}}}}
\nc{\bfj}{{\boldsymbol{\sl{j}}}}
\nc{\bfk}{{\boldsymbol{\sl{k}}}}
\nc{\bfl}{{\boldsymbol{\sl{l}}}}
\nc{\bfm}{{\boldsymbol{\sl{m}}}}
\nc{\bfn}{{\boldsymbol{\sl{n}}}}
\nc{\bfs}{{\boldsymbol{\sl{s}}}}
\nc{\bfr}{{\boldsymbol{\sl{r}}}}
\nc{\bft}{{\boldsymbol{\sl{t}}}}
\nc{\bfx}{{\boldsymbol{\sl{x}}}}
\nc{\bfz}{{\boldsymbol{\sl{z}}}}
\nc\bfmu{{\boldsymbol \mu}}
\nc\bfgl{{\boldsymbol \lambda}}
\nc\bfsi{{\boldsymbol \sigma}}
\nc\bfet{{\boldsymbol \eta}}
\nc\bfeta{{\boldsymbol \eta}}
\nc\bfeps{{\boldsymbol \varepsilon}}
\nc\bfone{{\bf 1}}

 \allowdisplaybreaks

\catcode`!=11
\let\!int\int \def\int{\displaystyle\!int}
\let\!lim\lim \def\lim{\displaystyle\!lim}
\let\!sum\sum \def\sum{\displaystyle\!sum}
\let\!sup\sup \def\sup{\displaystyle\!sup}
\let\!inf\inf \def\inf{\displaystyle\!inf}
\let\!cap\cap \def\cap{\displaystyle\!cap}
\let\!max\max \def\max{\displaystyle\!max}
\let\!min\min \def\min{\displaystyle\!min}
\let\!frac\frac \def\frac{\displaystyle\!frac}
\catcode`!=12

\let\oldsection\section
\renewcommand\section{\setcounter{equation}{0}\oldsection}

\allowdisplaybreaks

\DeclareMathOperator*{\sgn}{sgn}
\DeclareMathOperator*{\dep}{dep}
\DeclareMathOperator{\Li}{Li}

\DeclareMathOperator{\sih}{sih}

\DeclareMathOperator{\Reg}{Reg}

\nc\UU{\mbox{\bfseries U}}
\nc\FF{\mbox{\bfseries \itshape F}}
\nc\h{\mbox{\bfseries \itshape h}}\nc\dd{\mbox{d}}
\nc\g{\mbox{\bfseries \itshape g}}
\nc\xx{\mbox{\bfseries \itshape x}}

\nc\gl{{\lambda}}
\nc\gD{{\Delta}}
\nc\ga{{\alpha}}
\nc\gb{{\beta}}
 \nc{\gam}{{\gamma}}
 \nc{\gG}{{\Gamma}}
 \nc{\vep}{{\varepsilon}}
 \nc{\gs}{{\sigma}}
 \nc{\gth}{{\theta}}
 \nc{\gS}{{\Sigma}}
 \nc{\gf}{{\phi}}
 \nc{\gk}{{\kappa}}
 \nc{\gm}{{\mu}}
 \nc{\gM}{{M}}
 \nc{\gz}{{\zeta}}
 \nc{\tlg}{{\tilde{g}}}
 \nc{\tlb}{{\tilde{b}}}
 \nc{\tla}{{\tilde{a}}}
 \nc{\tlq}{{\tilde{q}}}
 \nc{\tlp}{{\tilde{p}}}
 \nc{\tlgf}{{\tilde{\gf}}}
 \nc{\tlh}{{\tilde{h}}}
 \nc{\tlk}{{\tilde{\gk}}}
 \nc{\tlgz}{{\tilde{\zeta}}}
 \nc{\gO}{{\Omega}}
 \nc{\sif}{{\mathcal S}}
 \nc{\gt}{{\tau}}
 \nc{\bt}{{\chi}}
 \nc{\tlt}{{\tilde{t}}}
 \nc{\tlgk}{{\tilde{\gk}}}
 \nc{\binn}{{\binom{2n}{n}}}

\def\N{\mathbb{N}}

\def\Q{\mathbb{Q}}

\def\ze{\zeta}

\def\xx{\left(\frac{1-x}{1+x} \right)}

\def\ol{\overline}
\nc\divg{{\text{div}}}
\theoremstyle{plain}
\newtheorem{thm}{Theorem}[section]

\newtheorem{cor}[thm]{Corollary}

\theoremstyle{definition}

\newtheorem{re}[thm]{Remark}
\newtheorem{ex}[thm]{Example}

\nc{\myone}{{1}}
\nc{\myO}{{\mathsf O}}
\nc{\myL}{{\mathsf L}}

\let\Im\relax

\DeclareMathOperator\Im{{{Im}}}
\DeclareMathOperator\sh{{{sh}}}
\DeclareMathOperator\ch{{{ch}}}
\let\th\relax
\DeclareMathOperator\th{{{th}}}
\DeclareMathOperator\cth{{{cth}}}
\DeclareMathOperator\sech{{{sech}}}
\DeclareMathOperator\csch{{{csch}}}

\begin{document}
\title{\bf Alternating Ap\'{e}ry-Type Series and \\ Colored Multiple Zeta Values of Level Eight} 
\author{
{Ce Xu${}^{a,}$\thanks{Email: cexu2020@ahnu.edu.cn,  ORCID 0000-0002-0059-7420.}\ \ and Jianqiang Zhao${}^{b,}$\thanks{Email: zhaoj@ihes.fr, ORCID 0000-0003-1407-4230.}}\\[1mm]
\small a. School of Mathematics and Statistics, Anhui Normal University, Wuhu 241002, PRC\\
\small b. Department of Mathematics, The Bishop's School, La Jolla, CA 92037, USA}

\date{}
\maketitle

\noindent{\bf Abstract.} Ap\'{e}ry-type (inverse) binomial series have appeared prominently in the calculations of Feynman integrals in recent years. In our previous work, we showed that a few large classes of the non-alternating Ap\'ery-type (inverse) central binomial series can be evaluated using colored multiple zeta values of level four (i.e., special values of multiple polylogarithms at fourth roots of unity) by expressing them in terms of iterated integrals. In this sequel, we shall prove that for several classes of the alternating versions we need to raise the level to eight. Our main idea is to adopt hyperbolic trigonometric 1-forms to replace the ordinary trigonometric ones used in the non-alternating setting.



\section{Introduction}
Significant progress has been made on the calculations of the $\varepsilon$-expansion of multiloop Feynman diagram
in the past quarter of a century. It turns out that special values of multiple polylogarithms have played indispensable
roles in these computations. Many experimental work emerged around the beginning of this century, e.g., see \cite{BorweinBrKa2001,DavydychevDe2001,DavydychevDe2004,JegerlehnerKV2003}, in which a special class of series
emerges. These infinite sums are often called \emph{Ap\'ery-type series} (or \emph{Ap\'ery-like sums})
because the simplest cases were used by Ap\'ery
in his celebrated proof of irrationality of $\ze(2)$ and $\ze(3)$ in 1979. More precisely, he showed that
\begin{equation*}
 \ze(2)=3\sum_{n\ge 1} \frac{1}{n^2\binn} \quad\text{and}\quad \ze(3)=\frac52\sum_{n\ge 1} \frac{(-1)^{n-1}}{n^3\binn}.
\end{equation*}

Motivated by Ap\'ery's proof, Leshchiner \cite{Leshchiner} discovered the following generalizations
involving the alternating Riemann zetas $\ze(\bar n)$ which are defined by
\begin{equation}\label{equ-altRiemann}
\ze(\bar n) :=\sum_{k=1}^\infty \frac{(-1)^k}{k^n}=(2^{1-n}-1)\ze(n) \quad\forall n\ge 2.
\end{equation}
Let $\gd_{j,1}$ denote the Kronecker symbol. For all $j\in\N$, put
$
A_j=1-\frac{\gd_{j,1}}{4},\  B_j=1+\frac{\gd_{j,1}}{4}.
$

\begin{thm} \label{thm:Leshchiner} \emph{(\cite{Leshchiner})}
For any $k\in\N$, we have
\begin{align*}
\ze(\ol{2k})=&\, \sum_{n\ge 1} \frac{2}{n^2\binn}\sum_{j=1}^{k} A_j\frac{(-1)^{k-j}\ze_{n-1}(2_{k-j})}{n^{2j-2}} , \\
\ze(2k+1)=&\, \sum_{n\ge 1} \frac{2(-1)^{n-1}}{n^3\binn}\sum_{j=1}^{k} B_j \frac{(-1)^{k-j} \ze_{n-1}(2_{k-j})}{n^{2j-2}}, \\
\sum_{n\ge 0} \frac{(-1)^n}{(2n+1)^{2k-1}}=&\, \sum_{n\ge 0} \frac{\binn}{4^{2n}}\sum_{j=1}^{k} A_j \frac{(-1)^{k-j}t_n(2_{k-j})}{n^{2j-1}} , \\
\sum_{n\ge0} \frac{1}{(2n+1)^{2k}}=&\, \sum_{n\ge 0} \frac{(-1)^n\binn}{4^{2n}}\sum_{j=1}^{k} B_j\frac{(-1)^{k-j} t_n(2_{k-j})}{(2n+1)^{2j}} ,
\end{align*}
where $2_p$ means the string of $2$'s with $p$ repetitions, $B_j$'s are Bernoulli numbers,
and $\ze_{n-1}$ and $t_n$ are defined by \eqref{defn-zetan} and \eqref{defn-tn}, respectively.
\end{thm}

We will call the Ap\'ery-type series a \emph{central binomial series} (resp.\ \emph{inverse central binomial series})
if the central binomial coefficient $\binn$ appears on the numerator (resp.\ denominator). It is remarkable that both types appear in Thm.~\ref{thm:Leshchiner} and both appear in the evaluations of Feynman integrals (see \cite{BorweinBrKa2001,DavydychevDe2001,DavydychevDe2004} for inverse binomial series and \cite{JegerlehnerKV2003} for binomial series).
Moreover, odd-indexed variations of both types appeared implicitly, too. See \cite[(1.1)]{DavydychevDe2004} and \cite[Remark 4.2]{XuZhao2022a} for the former
and \cite[(A.25)]{JegerlehnerKV2003} and \cite[Eq.~(1.3)]{XuZhao2022a} for the latter.

In our previous work \cite{XuZhao2021b,XuZhao2022a,XuZhao2022b} we have also studied both types and considered their even-odd-indexed variations. Our main result is that many such series can be expressed as $\Q$-linear combinations of the real and/or the imaginary part of colored multiple zeta values of level 4. We have further considered similar series with $\binn$ replaced by its square, in which case a possible extra factor of $1/\pi$ may appear for the binomial series. In this paper, we will focus primarily on the alternating versions.

We begin with some basic notations. Let $\N$ be the set of positive integers and $\N_0:=\N\cup \{0\}$.
A finite sequence $\bfk:=(k_1,\ldots, k_r)\in\N^r$ is called a \emph{composition}. We put
\begin{equation*}
 |\bfk|:=k_1+\cdots+k_r,\quad \dep(\bfk):=r,
\end{equation*}
and call them the \emph{weight} and the \emph{depth} of $\bfk$, respectively. If $k_1>1$, $\bfk$ is called \emph{admissible}.
For any $m,n,d\in\N$ with $m\ge n+d$ and any compositions $\bfs=(s_1,\dots,s_d)\in\N^d$, set
\begin{align*}
\ze_m(\bfs)_{n}:=&\, \sum_{m\ge n_1> n_2>\cdots>n_d> n} \frac{1}{n_1^{s_1}\cdots n_d^{s_d}},\\
\ze^\star_m(\bfs)_{n}:=&\, \sum_{m\ge n_1\ge\cdots\ge n_d>n}\frac{1}{n_1^{s_1}\cdots n_d^{s_d}},
\end{align*}
and define \emph{the multiple harmonic (star) sums} by
\begin{align}\label{defn-zetan}
\ze_n(\bfs):=\ze_n(\bfs)_0,\qquad
\ze^\star_n(\bfs):=\ze^\star_n(\bfs)_0.
\end{align}
Then the \emph{multiple zeta values} (MZVs) and the \emph{multiple zeta star values} (MZSVs) are defined by
\begin{align*}
\ze(\bfs):=\lim_{n\to \infty} \ze_n(\bfs)\quad\text{and}\quad
\ze^\star(\bfs):=\lim_{n\to \infty} \ze^\star_n(\bfs),
\end{align*}
respectively, and their $n$-tails are defined by
\begin{align*}
\ze(\bfs)_n:=\lim_{m\to \infty} \ze_m(\bfs)_n\quad\text{and}\quad
\ze^\star(\bfs)_n:=\lim_{m\to \infty} \ze^\star_m(\bfs)_n,
\end{align*}
respectively. We call $|\bfs|:=s_1+\cdots+s_d$ the \emph{weight} and $d$ the \emph{depth} of the corresponding values.
Kontsvitch observed that $\ze(\bfs)$ can be expressed using Chen's iterated integrals
(see \cite[Sec.~2.1]{Zhao2016} for a brief summary of this theory):
\begin{equation}\label{zeta-itIntExpression}
\ze(\bfs)=\int_0^1 \bfw(\bfs)  \quad\text{and}\quad \bfw(\bfs):=\ta^{s_1-1}\tx_1\cdots\ta^{s_d-1}\tx_1.
\end{equation}
Here we have put $\ta=dt/t$ and $\tx_\xi=dt/(\xi-t)$ for any $N$th roots of unity $\xi$. We will
call $\bfw(\bfs)$ the \emph{associated word} of $\bfs$.

In general, let $\bfs=(s_1,\ldots,s_d)\in\N^d$ and $\bfeta=(\eta_1,\dotsc,\eta_d)$ where $\eta_1,\dotsc,\eta_d$ are $N$th roots of unity. We can define the \emph{colored MZVs} (CMZVs) of level $N$ by
\begin{equation}\label{equ:defnMPL}
\Li_{\bfs}(\bfeta):=\sum_{n_1>\cdots>n_d>0}
\frac{\eta_1^{n_1}\dots \eta_d^{n_d}}{n_1^{s_1} \dots n_d^{s_d}},
\end{equation}
which converges if $(s_1,\eta_1)\ne (1,1)$ (see \cite{Racinet2002} and \cite[Ch. 15]{Zhao2016}), in which case we call $({\bfs};\bfeta)$ \emph{admissible}. The level two colored MZVs are often called \emph{Euler sums} or \emph{alternating MZVs}. In this case, namely, when $(\eta_1,\dotsc,\eta_d)\in\{\pm 1\}^r$ and $(s_1,\eta_1)\ne (1,1)$, we set
$\ze(\bfs;\bfeta)= \Li_\bfs(\bfeta)$. Further, to save space we put a bar on top of
$s_{j}$ if the corresponding $\eta_{j}=-1$, which is consistent with the notation in \eqref{equ-altRiemann}. For example,
\begin{equation*}
\ze(\bar2,3,\bar1,4)=\ze(2,3,1,4;-1,1,-1,1).
\end{equation*}
Moreover, CMZVs can be expressed by Chen's iterated integrals. In fact, for any complex number $z$
such that $(s_1,z\eta_1)\ne(1,1)$ we have
\begin{equation}\label{czeta-itIntExpression}
\Li_{\bfs}(z\eta_1,\eta_2,\dots,\eta_d)=\int_0^z \ta^{s_1-1}\tx_{\ga_1}\cdots\ta^{s_d-1}\tx_{\ga_d},
\end{equation}
where $\ga_j:=\prod_{i=1}^j \eta_i^{-1}=1/(\eta_1\cdots \eta_j)$ for all $j$.

In \cite{Akhilesh1,Akhilesh} Akhilesh discovered some very important and surprising connections between
MZVs and the following Ap\'ery-type series, or multiple Ap\'ery-like sums:
\begin{equation*}
\gs(\bfs;x):= \sum_{n_1> n_2>\cdots>n_d} {\binom{2n_1}{n_1}}^{-1} \frac{(2x)^{2n_1}}{(2n_1)^{s_1}\cdots (2n_d)^{s_d}}.
\end{equation*}
Here we have renormalized the sums to make them consistent with our previous work.
His ingenious idea is to study the $n$-tails (and more generally, double tails) of such series.
We reformulate one of his key result as follows to make it more transparent. Set
\begin{align*}
g_s(t)=g_s^+(t)=
\left\{
 \begin{array}{ll}
 \tan t\, dt, \quad & \hbox{if $s=1$;} \\
 dt \circ (\cot t\, dt)^{s-2}\circ dt, \quad & \hbox{if $s\ge 2$.}
 \end{array}
\right.
\end{align*}
Further, we set $\binom{0}{0}=1$,
\begin{align*}
 b_n(x)=4^n{\binn}^{-1}x^{2n}\quad \text{and}\quad b_n= b_n(1)=4^n{\binn}^{-1}\quad \forall n\ge 0.
\end{align*}

\begin{thm} \label{thm-gs-Akhilesh} \emph{(\cite[Thm.~4]{Akhilesh})}
For all $n\in\N_0$, $\bfs=(s_1,\dots,s_d)\in\N^d$ we have
\begin{align*}
\gs(\bfs;\sin y)_n:=&\, \sum_{n_1 > \cdots > n_d>n}
\frac{b_{n_1}(\sin y)}{(2n_1)^{s_1}\cdots (2n_d)^{s_d}} \\
=&\, \frac{d}{dy} \int_0^y g_{s_1}(t)\circ \cdots \circ g_{s_d}(t)\circ b_n(\sin t) \,dt.
\end{align*}
Here $y\in(-\pi/2,\pi/2)$ if $n_1=1$ and $y\in[-\pi/2,\pi/2]$ if $n_1\ge 2$.
\end{thm}

In this paper, motivated by a series of conjectures by Z.-W. Sun \cite{Sun2015}, and Leshchiner's
and Akhilesh's results above, we shall investigate the following \emph{alternating} Ap\'ery-type series.
For any $n\in\N$, $\bfeta=(\eta_1,\dots,\eta_d)\in\{\pm 1\}^d$, $\bfs=(s_1,\dots,s_d)\in\N^d$ and a complex variable $x$,
define
\begin{align*}
\gs(\bfs;\bfeta;x)_n:=\sum_{n_1>\dots>n_d>n}
 \frac{b_{n_1}(x)\eta_1^{n_1}\cdots \eta_d^{n_d}}{(2n_1)^{s_1}\cdots (2n_d)^{s_d}},
\end{align*}
which is called the $n$-\emph{tail} of the alternating Ap\'ery-type series
\begin{align*}
\gs(\bfs;\bfeta;x):=\gs(\bfs;\bfeta;x)_0.
\end{align*}
To save space we will put a bar on top of $s_j$ if the corresponding $\eta_j=-1$. For example,
\begin{align*}
 \gs(\bar3,4,\bar1;x)\,&=\sum_{n_1>n_2>n_3}
 \frac{b_{n_1}(x) (-1)^{n_1+n_3}}{(2n_1)^{3}(2n_2)^{4} (2n_3) } .
\end{align*}

{}From Leschiner's result we see that studying Ap\'ery-type series naturally leads to the following multiple $t$-harmonic sums
(and their star version). For any $m,n,d\in\N$ and compositions $\bfs=(s_1,\dots,s_d)\in\N^d$, define
\begin{align*}
t_m(\bfs)_{n}:=&\, \sum_{m\ge n_1> n_2> \cdots > n_d> n} \frac{1}{(2n_1-1)^{s_1}\cdots (2n_d-1)^{s_d}},\\
t^\star_m(\bfs)_{n}:=&\, \sum_{m\ge n_1\ge \cdots \ge n_d> n}
\frac{1}{(2n_1-1)^{s_1}\cdots (2n_d-1)^{s_d}},
\end{align*}
and the \emph{multiple $t$-harmonic (star) sums}
\begin{align}\label{defn-tn}
t_n(\bfs):=t_n(\bfs)_0,\quad\text{and}\quad
t^\star_n(\bfs):=t^\star_n(\bfs)_0,
\end{align}
respectively. Correspondingly, the \emph{multiple $t$-harmonic (star) values} are the infinite sums
\begin{align*}
t(\bfs):=\lim_{n\to\infty} t_n(\bfs),\quad\text{and}\quad
t^\star(\bfs):=\lim_{n\to\infty} t^\star_n(\bfs).
\end{align*}

We will use heavily the hyperbolic trigonometric functions throughout this paper. To fix notation, we put
\begin{align*}
 \sh x = -i \sin (i x)=\frac{e^x - e^{-x}} {2},\quad
 \ch x = \cos (i x)=\frac{e^x + e^{-x}} {2},\quad
 \th x = -i \tan(i x)=\frac{e^x - e^{-x}}{e^x + e^{-x}},\\
 \cth x = i \cot(i x)=\frac{e^x + e^{-x}} {e^x - e^{-x}},\quad
 \sech x = \sec (i x)=\frac{2} {e^x + e^{-x}},\quad
 \csch x = i \csc (i x)=\frac{2} {e^x - e^{-x}}.
\end{align*}

We will extend Chen's iterated integrals by combining 1-forms and functions as follows.
For any $r\in\N$, 1-forms $f_1(t)\,dt,\dots,f_{r+1}(t)\,dt$ and functions $F_1(t),\dots,F_r(t)$, define
recursively
\begin{align*}
&\, \int_0^1 \big( f_1(t)\,dt+F_1(t)\big)\circ \cdots\circ \big( f_r(t)\,dt+F_r(t)\big)\circ f_{r+1}(t)\,dt\\
:=&\, \int_0^1 \big( f_1(t)\,dt+F_1(t)\big)\circ \cdots\circ \big( f_{r-1}(t)\,dt+F_{r-1}(t)\big)\circ f_r(t)\,dt \circ f_{r+1}(t)\,dt\\
+&\, \int_0^1 \big( f_1(t)\,dt+F_1(t)\big)\circ \cdots\circ \big( f_{r-1}(t)\,dt+F_{r-1}(t)\big)\circ \big(F_r(t)f_{r+1}(t)\big)\,dt.
\end{align*}

\section{Alternating Ap\'ery-type inverse central binomial series}\label{sec:invBinn}
Define the hyperbolic counterpart of $g_s(t)$ by
\begin{align*}
\tlg_s(t)=\tlg_s^+(t)=
\left\{
 \begin{array}{ll}
 \th t\, dt, \quad & \hbox{if $s=1$;} \\
 dt \circ (\cth t\, dt)^{s-2}\circ dt, \quad & \hbox{if $s\ge 2$.}
 \end{array}
\right.
\end{align*}
Throughout the paper we put $\nu:=\sqrt2+1$.

\begin{thm} \label{thm-gsbar1-ItInt}
Set $\psi=\sh^{-1} 1=\log\nu$. For all $n\in\N_0$, $\bfs=(s_1,\dots,s_d)\in\N^d$ we have
\begin{align*}
\gs(\bfs;i\sh y)_n=&\,(-1)^d \frac{d}{dy}\int_0^y \tlg_{s_1}(t)\circ \cdots \circ \tlg_{s_d}(t)\circ b_n(i\sh y) \,dt.
\end{align*}
Here $y\in(-\psi,\psi)$ if $s_1=1$ and $y\in[-\psi,\psi]$ if $s_1\ge 2$.
\end{thm}

\begin{proof}
Applying the substitution $y\to iy$ in Thm.~\ref{thm-gs-Akhilesh} we obtain immediately
\begin{align*}
\gs(\bfs;i\sh y)_n:=&\, \sum_{n_1 > \cdots > n_d>n}
\frac{b_{n_1}(i\sh y)}{(2n_1)^{s_1}\cdots (2n_d)^{s_d}} \\
=&\, -i\frac{d}{dy} \int_0^y g_{s_1}(it)\circ \cdots \circ g_{s_d}(it)\circ b_n(i\sh t) \,d(it),
\end{align*}
where we have used the fact that $\sin it=i\sh t$ and $\cos it=\ch t$. Note that
\begin{align*}
g_s(i t)= &
\left\{
 \begin{array}{ll}
 \tan (it)\, d(it), \quad & \hbox{if $s=1$;} \\
 d(it) \circ (\cot (it)\, d(it))^{s-2}\circ d(it), \ \, \quad\quad & \hbox{if $s\ge 2$,}
 \end{array}
\right. \\
= &
\left\{
 \begin{array}{ll}
 -\th t\, dt, \quad & \hbox{if $s=1$;} \\
 -dt \circ (\cth t\, dt)^{s-2}\circ dt, \qquad\qquad \qquad & \hbox{if $s\ge 2$.}
 \end{array}
\right.
\end{align*}
The theorem follows immediately.
\end{proof}

Similarly to the original form in (\cite[Thm.~4]{Akhilesh}),
setting $\bfw(\bfs)=\tx_{\eps_1}\dots \tx_{\eps_w}$ (see \eqref{zeta-itIntExpression}) yields
\begin{align*}
\gs(\bfs; i\sh y)_{n}=(-1)^d
\th^{\eps_1}y \int_0^y \th^{\eps_1+\eps_2-1} t\, dt \cdots \th^{\eps_{w-1}+\eps_w-1} t\, dt b_n(i\sh t) \,dt.
\end{align*}

\begin{cor} \label{cor:gs12pANDgs2p}
For all $p\in\N$ we have
\begin{align*}
\gs(1,2_{p};i\sh y)&\,=\gs(\bar1,2_p;\sh y)=(-1)^{p+1} \th y \int_0^{y} (dt)^{2p+1}=\frac{(-1)^{p+1}}{(2p+1)!} y^{2p+1}\th y, \\
\gs(2_{p};i\sh y)&\,=\gs(\bar2,2_{p-1};\sh y)=(-1)^p \int_0^y (dt)^{2p}=\frac{(-1)^p}{(2p)!} y^{2p}.
\end{align*}
\end{cor}

\begin{re}
Taking $p=0$ and $p=1$ in the first equation of Cor.~\ref{cor:gs12pANDgs2p}, we
can recover the two equations (C.8) and (C.16) in \cite{DavydychevDe2004}.
Similarly, taking $p=0$ and $p=1$ in the second equation above, we
can obtain (C.9) and (C.17) in \cite{DavydychevDe2004}.
\end{re}

\begin{ex} For $j=1,2,3,4$ we have
\begin{align*}
\sh^{-1}\Big(\frac{\sqrt{j}}2\Big)=\log\Big(\frac{\sqrt{j}+\sqrt{j+4}}{2}\Big).
\end{align*}
Thus for all $p\ge 0$ we get
\begin{align*}
\gs\Big(1,2_{p};\frac{\sqrt{j}}2 i\Big)=\gs\Big(\bar1,2_p;\frac{\sqrt{j}}2 \Big)
=&\, \frac{(-1)^{p+1}\sqrt{j}}{\sqrt{j+4}} \int_0^{\sh^{-1}(\sqrt{j}/2)} (dt)^{2p+1}\\
 =&\, \frac{(-1)^{p+1}\sqrt{j}}{(2p+1)!\sqrt{j+4}} \log^{2p+1}\Big(\frac{\sqrt{j}+\sqrt{j+4}}{2}\Big).
\end{align*}
Similarly, for all $p\ge 1$,
\begin{align*}
 \gs\Big(2_{p};\frac{\sqrt{j}}2 i \Big)=\gs\Big(\bar2,2_{p-1};\frac{\sqrt{j}}2\Big)
=&(-1)^p \int_0^{\sh^{-1}(\sqrt{j}/2)} (dt)^{2p}=\frac{(-1)^p}{(2p)!} \log^{2p}\Big(\frac{\sqrt{j}+\sqrt{j+4}}{2}\Big).
\end{align*}
\end{ex}

\begin{re}
Kalmykov and his collaborators encountered and studied similar sums when computing Feynman integrals
in \cite{DavydychevDe2004,KalmykovKniehl2007,KalmykovWardYost2007}. By the stuffle relations it is easy to see that
those sums can be expressed as $\Q$-linear combinations of $\gs$'s and their odd-indexed variations, but not vice versa in general.
\end{re}

We now turn to the general alternating case. Define
\begin{align*}
\om_0:=&\frac{dt}{t},\quad \om_{\pm 1}:=\frac{dt}{\sqrt{1\mp t^2}},\quad
\om_{\pm 2}:=\frac{t\,dt}{1\mp t^2} ,\quad \om_{\pm 4}:=\frac{t\, dt}{\sqrt{1-t^4}}.
\end{align*}

\begin{thm} \label{thm-altApery-level8-NonTrig}
For any $d\in \N$, $\bfs=(s_1,\dots,s_d)\in\N^d$ and $\bfeta=(\eta_1,\dots,\eta_d)\in \{\pm1\}^d$,
put $\bfeta_j=\prod_{k=1}^j \eta_k$ and $\sgn(\bfeta)=\prod_{k=1}^d \bfeta_k$. Set $\sih_1 y=\sin y$ and $\sih_{-1} y=\sh y$.
Put $ b_n^{\pm}(t)={\binn}^{-1} (\pm 4t^2)^{n}$. Then we have
\begin{align}\label{eqn-thm-altApery-trig}
 \gs(\bfs;\bfeta; \sih_{\eta_1} y)_n
 =&\,\sgn(\bfeta) \frac{d}{dy} \int_0^{\sih_{\eta_1} y} \gG_{s_1}^{\bfeta_0,\bfeta_1}(t) \circ\cdots\circ\gG_{s_d}^{\bfeta_{d-1},\bfeta_d}(t)\circ b_n^{\sgn(\bfeta_d)}(t) \om_{\bfeta_d},\\
 \gs(\bfs;\bfeta; x)_n
 =&\,\sgn(\bfeta) \sqrt{1-\eta_1x^2} \frac{d}{dx} \int_0^{x} \gG_{s_1}^{\bfeta_0,\bfeta_1}(t) \circ\cdots\circ\gG_{s_d}^{\bfeta_{d-1},\bfeta_d}(t)\circ b_n^{\sgn(\bfeta_d)}(t) \om_{\bfeta_d}, \label{eqn-thm-altApery}
\end{align}
where $\bfeta_0=\bfeta_1$ and
\begin{equation*}
 \gG_s^{a,b} (t)=
 \left\{
\begin{array}{ll}
 \om_{3b-a} {}_{ \phantom{\textstyle 1 } } & \hbox{if $s=1$;} \\
 \om_a\om_0^{s-2}\om_b \qquad & \hbox{if $s\ge 2$.}
 \end{array}
 \right.
\end{equation*}
Here if $\eta_1=1$ then $\gs(\bfs;\bfeta; 1)_n$ is the limit $x\to 1^{-}$ on the right-hand side of \eqref{eqn-thm-altApery}.
\end{thm}

\begin{re} Note that in general $\sgn(\bfeta)=\eta_1^d\cdots \eta_{d-1}^2\eta_d \ne\prod_{k=1}^d \eta_k$.
\end{re}

\begin{proof} Equation \eqref{eqn-thm-altApery} is an easy consequence of \eqref{eqn-thm-altApery-trig} which
we will prove by induction on $d$. When $d=1$, from Thm.~\ref{thm-gs-Akhilesh} and using
the change of variables $t\to \sin^{-1} t$ we see that
\begin{align*}
\gs(1;\sin y)_n =&\, \frac{d}{dy} \int_0^y \tan t\,dt \, b_n^{+}(\sin t) dt
 = \frac{d}{dy} \int_0^{\sin y} \om_2\, b_n^{+}(t) \om_{1},\\
\gs(s;\sin y)_n=&\, \frac{d}{dy} \int_0^y dt \bigg(\frac{dt}{\tan t}\bigg)^{s-2} dt\, b_n^{+}(\sin t)\, dt
 = \frac{d}{dy} \int_0^{\sin y} \om_1 \om_0^{s-2}\om_{1} \, b_n^{+}(t) \om_1.
\end{align*}
Similarly, from Thm.~\ref{thm-gsbar1-ItInt} and using
the change of variables $t\to \sh^{-1} t$ we get
\begin{align*}
\gs(\bar1;\sh y)_n =&\, \frac{d}{dy} \int_0^y \th t\,dt \, b_n^{-}(\sh t) dt
 = -\frac{d}{dy} \int_0^{\sh y} \om_{-2}\, b_n^{-}(t) \om_{-1},\\
\gs(\bar{s};\sh y)_n
 =&\, -\frac{d}{dy} \int_0^y dt \bigg(\frac{dt}{\th t}\bigg)^{s-2} dt\, b_n^{-}(\sh t)\, dt
 = -\frac{d}{dy} \int_0^{\sh y} \om_{-1} \om_0^{s-2}\om_{-1} \, b_n^{-}(t) \om_{-1}.
\end{align*}
Hence the depth $d=1$ case is proved.

Assume the theorem holds when the depth is $d-1$ for some $d\ge2$. Write $s=s_d$, $\bfs'=(s_1,\dots,s_{d-1})$ and $\bfeta'=(\eta_1,\dots,\eta_{d-1})$. By the definition
\begin{align*}
&\, \gs(\bfs;\bfeta; \sih_{\eta_1} y)_n=\sum_{k>n} \frac{\gs(\bfs';\bfeta'; \sih_{\eta_1} y)_k \eta_d^k}{(2k)^s}\\
=&\,\sgn(\bfeta') \frac{d}{dy} \int_0^{\sih_{\eta_1} y} \gG_{s_1}^{\bfeta_0,\bfeta_1}(t) \circ\cdots\circ\gG_{s_d}^{\bfeta_{d-2},\bfeta_{d-1}}(t)\circ
\sum_{k>n} \frac{\eta_d^k b_k^{\sgn(\bfeta_{d-1})}(t)}{(2k)^s} \, \om_{\bfeta_{d-1}}
\end{align*}
by the inductive assumption. But
\begin{align*}
 \sum_{k>n} \frac{\eta_d^k b_k^{\sgn(\bfeta_{d-1})}(t)}{(2k)^s}
 =\sum_{k>n} {\binn}^{-1} \frac{(\bfeta_d 4t^2)^{k}}{(2k)^s}
 =\gs(s;\bfeta_d;t)_n.
\end{align*}
If $\bfeta_d=1$ then setting $\gt=\sin^{-1} t$ and applying the change of variables $u=\sin^{-1} z$
in the iterated integral in Thm.~\ref{thm-gs-Akhilesh} we have
\begin{align}
\gs(s;\bfeta_d;t)_n=\gs(s; \sin \gt)_n = &\, \frac{d}{d\gt}\int_0^{\gt} g_s(u) \circ b_n^{+}(\sin u) \, du \notag \\
= &\, \bigg(\frac{d\gt}{dt}\bigg)^{-1}\frac{d}{dt} \int_0^t \gG_s^{1,1}(z) \circ b_n^{+}(z) \om_1(z) \notag \\
= &\,
\left\{
 \begin{array}{ll}
 \frac{t}{\sqrt{1-t^2}} \int_0^t b_n^{+}(z) \om_1(z) , & \hbox{if $s=1$;} \\
 \int_0^t \om_0(z)^{s-2} \om_1(z) \circ b_n^{+}(z) \om_1(z), \qquad & \hbox{if $s\ge 2$.}
 \end{array}
\right.
\label{keepSign2}
\end{align}
Noticing that
\begin{align*}
\frac{t}{\sqrt{1-t^2}} \om_{\bfeta_{d-1}}=
\left\{
 \begin{array}{ll}
 \frac{t\, dt}{1-t^2}=\om_2, & \hbox{if $\bfeta_{d-1}=1$;} \\
 \frac{t\, dt}{\sqrt{1-t^4}}=\om_4, \qquad & \hbox{if $\bfeta_{d-1}=-1$,}
 \end{array}
\right.
\end{align*}
we deduce that
\begin{align}\label{keepSign3}
\sum_{k>n} \frac{\eta_d^k b_k^{\sgn(\bfeta_{d-1})}(t)}{(2k)^s}\, \om_{\bfeta_{d-1}}
=&\, \gs(s;\bfeta_d;t)_n\, \om_{\bfeta_{d-1}} \nonumber\\
=&\, \left\{
 \begin{array}{ll}
 \om_{3-\bfeta_{d-1}} b_n^{+}(t)\om_1, & \hbox{if $s=1$;} \\
 \om_{\bfeta_{d-1}} \om_0^{s-2} \om_1 b_n^{+}(t)\om_1,\qquad & \hbox{if $s\ge2$.}
 \end{array}
\right. \nonumber\\
=&\, \bfeta_d \gG_{s_d}^{\bfeta_{d-1},\bfeta_d}(t)\circ b_n^{\sgn(\bfeta_d)}(t).
\end{align}
This completes the induction proof for the case $\bfeta_d=1$.
If $\bfeta_d=-1$ then we only need to modify the above proof slightly by replacing all the trigonometric functions
by their hyperbolic counterpart, replacing $ b_n^+$ by $ b_n^-$,
replacing $\om_j$ by $\om_{-j}$ for $j=1,2,4$, and keeping an extra negative sign in
the front of \eqref{keepSign2}--\eqref{keepSign3}. Thus
\begin{align*}
\sum_{k>n} \frac{(-1)^k f_k^{\sgn(\bfeta_{d-1})}(t)}{(2k)^s} \, \om_{\bfeta_{d-1}}
= &\,
\left\{
 \begin{array}{ll}
 - \om_{-3-\bfeta_{d-1}} b_n^{-}(t)\om_{-1} , & \hbox{if $s=1$;} \\
 - \om_{\bfeta_{d-1}} \om_0^{s-2} \om_{-1} b_n^{-}(t)\om_{-1} ,\qquad & \hbox{if $s\ge2$.}
 \end{array}
\right.
\end{align*}
Combining the two cases we see that for $\eta_d=\pm 1$
\begin{align*}
\sum_{k>n} \frac{\eta_d^k f_k^{\sgn(\bfeta_{d-1})}(t)}{(2k)^s} \, \om_{\bfeta_{d-1}}
=&\,
\left\{
 \begin{array}{ll}
 \bfeta_d \om_{3\bfeta_d-\bfeta_{d-1}} b_n^{\sgn(\bfeta_d)}(t)\om_{\bfeta_d} , & \hbox{if $s=1$;} \\
 \bfeta_d \om_{\bfeta_{d-1}} \om_0^{s-2} \om_{\bfeta_d} b_n^{\sgn(\bfeta_d)}(t)\om_{\bfeta_d} ,\qquad & \hbox{if $s\ge2$.}
 \end{array}
\right. \nonumber\\
=&\, \bfeta_d \gG_{s_d}^{\bfeta_{d-1},\bfeta_d}(t)\circ b_n^{\sgn(\bfeta_d)}(t).
\end{align*}
This concludes the proof of the theorem.
\end{proof}

\begin{thm}\label{thm-AltAperyLevel8}
Let $d\in \N$, $\bfs=(s_1,\dots,s_d)\in\N^d$ and $\bfeta=(\eta_1,\dots,\eta_d)\in \{\pm1\}^d$. If
$(s_1,\eta_1)\ne (1,1)$ then
\begin{align*}
\gs(\bfs;\bfeta;1)=\sum_{n_1>\dots>n_d>0}
 \frac{b_{n_1}\eta_1^{n_1}\cdots \eta_d^{n_d}}{(2n_1)^{s_1}\cdots (2n_d)^{s_d}}\in\CMZV^8_{|\bfs|}\otimes\Q[i,\sqrt{2}],
\end{align*}
where $\CMZV_{|\bfs|}^8$ is the $\Q$-span of all CMZVs of weight $|\bfs|$ and level $8$.
\end{thm}
\begin{proof}
This follows immediately from Thm.~\ref{thm-altApery-level8-NonTrig} by using the change of variables
\begin{align}\label{equ:changeLevel8}
 t\to \frac{\sqrt{2}t}{\sqrt{1+t^4}}.
\end{align}
Indeed, let $\mu_j=\exp((2j-1)\pi i/4)$ ($j=1,\dots,4$) be the four 8th roots of unity satisfying $x^4+1=0$.
Under \eqref{equ:changeLevel8} we get
\begin{alignat}{6}
\om_0:=&\frac{dt}{t} & \, \quad\to & \quad \quad \frac{(1-t^4)\, dt}{t(1+t^4)}
 &=&\, \ta+\frac12 \sum_{j=1}^4 \tx_{\mu_j}, \label{rho:ChangVar0}\\
\om_1:=&\frac{dt}{\sqrt{1-t^2}} & \, \quad\to & \quad \frac{\sqrt{2} (1+t^2) dt}{1+t^4}
 &=&\, \frac{\sqrt{2}}{4}\sum_{j=1}^4 (\mu_j+\mu_j^3)\tx_{\mu_j} , \label{rho:ChangVar1}\\
\om_{-1}:=&\frac{dt}{\sqrt{1+t^2}} & \, \quad\to & \quad \frac{\sqrt{2} (1-t^2) dt}{1+t^4}
 &=&\, \frac{\sqrt{2}}{4}\sum_{j=1}^4 (\mu_j-\mu_j^3)\tx_{\mu_j} , \label{rho:ChangVar-1}\\
\om_2:=&\frac{t\,dt}{1-t^2} & \, \quad\to & \quad \frac{2 (1+t^2) dt}{(1-t^2)(1+t^4)}
 &=&\, \tx_1+\tx_{-1}- \frac12\sum_{j=1}^4 \tx_{\mu_j} , \label{rho:ChangVar2}\\
\om_{-2}:=&\frac{t\,dt}{1+t^2} & \, \quad\to & \quad \frac{2 (1-t^2) dt}{(1+t^2)(1+t^4)}
 &=&\, \frac12\sum_{j=1}^4 \tx_{\mu_j}-\tx_i-\tx_{-i}, \label{rho:ChangVar-2}\\
\om_{-4}:=\om_4=&\frac{t\, dt}{\sqrt{1-t^4}}& \, \quad\to & \qquad \ \frac{2 t\, dt}{1+t^4}
 &=&\, \frac{1}{2}\sum_{j=1}^4 \mu_j^2\tx_{\mu_j}. \label{rho:ChangVar4}
\end{alignat}
This completes the proof of the theorem.
\end{proof}

\begin{ex}\label{ex-bnOverBar2nBar2n}
By the ideas used in the proof of Thm.~\ref{thm-AltAperyLevel8}, we can compute
\begin{align}\label{equ-gs_1_1}
\gs(\bar1,\bar1;1)= &\,-\frac{1}{\sqrt2}\int_0^1\om_4\om_1
= -\frac{1}{8}\int_0^1 \left(\sum_{j=1}^4 \mu_j^2\tx_{\mu_j}\right)\left( \sum_{k=1}^4 (\mu_k+\mu_k^3)\tx_{\mu_k}\right)\\
= &\,-\frac18\sum_{j,k=1}^4 \mu_j^2(\mu_k+\mu_k^3) \int_0^1 \tx_{\mu_j} \tx_{\mu_k}\\
= &\,-\frac18\sum_{j,k=1}^4 \mu_j^2(\mu_k+\mu_k^3) \Li_{1,1}(\mu_j^{-1}, \mu_j/\mu_k)\approx -0.5346431875726234
\end{align}
by using Au's Mathematica package \cite{Au2020}. Similarly,
\begin{align*}
 \gs(\bar2,\bar1;1)=&\, -\int_0^1 \om_{-1} \om_4 \om_1
= -\frac{1}{16}\int_0^1
 \left(\sum_{j=1}^4 (\mu_j-\mu_j^3)\tx_{\mu_j} \right)
 \left(\sum_{l=1}^4 \mu_l^2 \tx_{\mu_l}\right)
 \left( \sum_{k=1}^4 (\mu_k+\mu_k^3)\tx_{\mu_k}\right)\\
= &\,-\frac{1}{16} \sum_{j,k,l=1}^4 (\mu_j+\mu_j^3)\mu_l^2(\mu_k-\mu_k^3) \Li_{1,1,1}(\mu_j^{-1}, \mu_j/\mu_l, \mu_l/\mu_k)\approx 0.0851511799,\\
 \gs(\bar2,1;1)=&\, \int_0^1 \om_{-1} \om_{-2} \om_{-1}
= \frac{1}{8}\int_0^1
 \sum_{j=1}^4 (\mu_j-\mu_j^3)\tx_{\mu_j}
 \left(\frac12\sum_{l=1}^4 \tx_{\mu_l}-\tx_i-\tx_{-i}\right)
 \sum_{k=1}^4 (\mu_k-\mu_k^3)\tx_{\mu_k} \\
= &\,\frac{1}{16} \sum_{j,k,l=1}^4 (\mu_j-\mu_j^3)(\mu_k-\mu_k^3) \Li_{1,1,1}(\mu_j^{-1}, \mu_j/\mu_l,\mu_l/\mu_k) \\
&\, -\frac{1}{8}\sum_{j,k=1}^4\sum_{\eps=\pm i} (\mu_j-\mu_j^3)(\mu_k-\mu_k^3)\Li_{1,1,1}(\mu_j^{-1}, \mu_j/\eps,\eps/\mu_k)\\
\approx &\, 0.045805888486699.
\end{align*}
We have checked these numerically by computing the series directly.
\begin{align*}
\end{align*}
\end{ex}

\begin{cor}\label{cor:algPts}
Let $d\in \N$, $\bfs=(s_1,\dots,s_d)\in\N^d$ and $\bfeta=(\eta_1,\dots,\eta_d)\in \{\pm1\}^d$.
For every real algebraic point $x$ with $|x|\le 1$ the value
\begin{align*}
\gs(\bfs;\bfeta;x),
\end{align*}
if it exists, can be expressed as a $\Q[i,\sqrt{2},\sqrt{1+\eta_1x^2}]$-linear combination of the multiple polylogarithms
evaluated at algebraic points.
\end{cor}
\begin{proof} Suppose $x\ne 0$.
By Thm.~\ref{thm-altApery-level8-NonTrig}, up to a factor of $\sqrt{1+\eta_1x^2}$ in front
$\gs(\bfs;\bfeta;x)$ can be expressed as
\begin{align}\label{equ-oms-1forms}
 \int_0^x \Big[\om_0, \om_{\pm 1},\om_{\pm 2},\om_4\Big]_{|\bfs|},
\end{align}
where $\Big[\om_0, \om_{\pm 1},\om_{\pm 2},\om_4\Big]_{|\bfs|}$ is an iteration of 1-forms of length $|\bfs|$.
Therefore, after applying the change of variables \eqref{equ:changeLevel8} we see that \eqref{equ-oms-1forms}
is transformed to a $\Q[i,\sqrt{2}]$-linear combination of iterated integrals of the form
\begin{align*}
 \int_0^{t(x)} \Big[\tx_0, \tx_\mu: \mu^8=1\Big]_{|\bfs|}, \quad \text{where}\quad t(x)=\frac{x^2-\sqrt{1-x^4}}{x^2}.
\end{align*}
Note $t(x)\to x$ under the change of variables \eqref{equ:changeLevel8}. The corollary follows from
\eqref{czeta-itIntExpression} immediately.
\end{proof}

\begin{ex} For $j=1,2,3,4$ we have the explicit evaluations
\begin{align*}
\gs\Big(\bar1;\frac{\sqrt{j}}2\Big)= &\,-\sqrt{\frac{j}{4+j}}\int_0^{\sqrt{j}/2}  \om_{-1}
=-\sqrt{\frac{j}{4+j}}\log\Big( \frac{\sqrt{j}+\sqrt{4+j}}{2}\Big).
\end{align*}
Using the idea of \eqref{equ-gs_1_1} we see that for all $j=1,2,3,4$ we have
\begin{align*}
\gs\Big(\bar1,\bar1;\frac{\sqrt{j}}2\Big)
= &\,-\sqrt{\frac{j}{4+j}}\int_0^{\sqrt{j}/2} \om_4 \om_{-1}
= -\frac{\sqrt{2j}}{8\sqrt{4+j}}\sum_{j,k=1}^4 \mu_j^2(\mu_k+\mu_k^3)   \int_0^{} \tx_{\mu_j} \tx_{\mu_k}\\
=&\,-\frac{\sqrt{2j}}{8\sqrt{4+j}}\sum_{j,k=1}^4 \mu_j^2(\mu_k+\mu_k^3) \Li_{1,1}\Big(c(j)\mu_j^{-1}, \mu_j/\mu_k\Big),
\end{align*}
where $c(j)=\sqrt{(4-\sqrt{16-j^2})/j}$ such that $c(j)\to \sqrt{j}/2$ under \eqref{equ:changeLevel8}.
We can similarly express $\gs(1,\bar1,\bar1;\sqrt{j}/2)$.
By Maple computation we find the numerical evaluations
\begin{align}
\gs(1,\bar1,\bar1;1/2)= &\,-\frac{\sqrt{3}}{3}\int_0^{1/2} \om_4^2 \om_1\approx -0.001257459248252,\nonumber\\
\gs(1,\bar1,\bar1;\sqrt{2}/2)= &\,-\int_0^{\sqrt{2}/2} \om_4^2 \om_1\approx -0.013713567545998,\nonumber\\
\gs(1,\bar1,\bar1;\sqrt{3}/2)= &\,-\sqrt{3}\int_0^{\sqrt{3}/2} \om_4^2 \om_1\approx -0.08102265305753797,\nonumber\\
\gs(\bar1,\bar1;1/2)= &\,-\frac{1}{\sqrt{5}}\int_0^{1/2} \om_4 \om_{1}\approx -0.019408779689355473,\nonumber\\
\gs(\bar1,\bar1;\sqrt{2}/2)= &\,-\frac{1}{\sqrt{3}}\int_0^{\sqrt{2}/2} \om_4 \om_{1}\approx -0.07667150401885149,\nonumber\\
\gs(\bar1,\bar1;\sqrt{3}/2)= &\,-\frac{\sqrt{3}}{\sqrt{7}}\int_0^{\sqrt{3}/2} \om_4 \om_{1}\approx -0.1845412608250132,\nonumber\\
\gs(\bar1,\bar1;1)= &\,-\frac{1}{\sqrt{2}}\int_0^1 \om_4 \om_{1}\approx -0.5346431875726234. \label{equ:gs-1-1}
\end{align}
The above have been verified numerically by directly computing the series and integrals separately. For
the last equation also see Example \ref{ex-bnOverBar2nBar2n}.
\end{ex}

\section{Two odd variations of alternating Ap\'ery-type inverse central binomial series}
In this section, we consider variations of the alternating Ap\'ery-type inverse binomial
series studied in section \ref{sec:invBinn}
by restricting the summation indices to odd numbers only. Recall that we have the following result.
Define the 1-forms
\begin{align}\label{defn-h}
h_{s}(t):=
\left\{
 \begin{array}{ll}
 2\csc 2t\, dt, & \hbox{if $s=1$;} \\
 \csc t\,dt \circ (\cot t\, dt)^{s-2}\circ \csc t\, dt, \qquad& \hbox{if $s\ge 2$.}
 \end{array}
\right.
\end{align}

\begin{thm} \emph{(\cite[Thm.~2.3]{XuZhao2022a})} \label{thm-1stVariantOld}
For all $n\in\N_0$, $\bfs=(s_1,\dots,s_d)\in\N^d$ we have the tail
\begin{align}\label{equ-thm-MtV-ItInt}
 \gt^\star(\bfs;\sin y)_n
 := & \, \sum_{n_1\ge \cdots \ge n_d\ge n}
\frac{b_{n_1}(\sin y)}{(2n_1+1)^{s_1}\cdots (2n_d+1)^{s_d}}
=\frac{d}{dy} \int_0^y h_{s_1}\circ \cdots h_{s_d}\circ b_{n}(\sin t) \,dt.
\end{align}
\end{thm}

Define the hyperbolic counterpart of $h_s(t)$ by
\begin{align*}
\tlh_s(t)=\tlh_s^+(t)=
\left\{
 \begin{array}{ll}
 2\csch 2t\, dt, & \hbox{if $s=1$;} \\
 \csch t\,dt \circ (\cth t\, dt)^{s-2}\circ \csch t\, dt, \qquad& \hbox{if $s\ge 2$.}
 \end{array}
\right.
\end{align*}

\begin{thm} \label{thm-1stVariant}
Set $\psi=\sh^{-1} 1=\log\nu$. For all $n\in\N_0$, $\bfs=(s_1,\dots,s_d)\in\N^d$ we have
\begin{align*}
 \gt^\star(\bfs; i\sh y)_n=&\, \frac{d}{dy}\int_0^y \tlh_{s_1}\circ \cdots \circ \tlh_{s_d}\circ b_n(i\sh t) \,dt.
\end{align*}
Here $y\in(-\psi,\psi)$ if $n_1=1$ and $y\in[-\psi,\psi]$ if $n_1\ge 2$.
\end{thm}

\begin{proof}
Applying the substitution $y\to iy$ in Thm.~\ref{thm-1stVariantOld}
we obtain the theorem immediately since $\csc it=-i\csch t$ and $\sech it=\sech t$.
\end{proof}

We now turn to another odd variation. Define the 1-forms
\begin{align}\label{defn-gk}
\gk_{s}(t)=
\left\{
 \begin{array}{ll}
 \sin t\, dt\, \csc t\,dt+ \tan t\, dt,\quad & \hbox{if $s=1$;} \\
 \sin t\, dt (\cot t\,dt+1) (\cot t\,dt)^{s-2} \csc t\, dt, \phantom{\frac12} \quad & \hbox{if $s\ge 2$.}
 \end{array}
\right.
\end{align}

\begin{thm} \emph{(\cite[Thm.~3.1]{XuZhao2022a})} \label{thm-2ndVariantOld}
For all $n\in\N_0$ and $\bfs=(s_1,\dots,s_d)\in\N^d$ the tail
\begin{align*}
\chi(\bfs;\sin y)_n:=\sum_{n_1> \cdots>n_d>n} \frac{ b_{n_1}(\sin y)}{(2n_1-1)^{s_1}\cdots (2n_d-1)^{s_d}}
= \frac{d}{d y} \int_0^y \gk_{s_1}\circ\cdots \circ \gk_{s_d}\circ b_{n}(\sin t) \,dt .
\end{align*}
In the above $y\in[-\pi/2,\pi/2]$ if $s_1>1$ and $y\in(-\pi/2,\pi/2)$ if $s_1=1$.
\end{thm}

Define the hyperbolic counterpart of $\gk_s(t)$ by
\begin{align*}
\tlk_s(t)=\tlk_s^+(t)=
\left\{
 \begin{array}{ll}
 -\sh t\, dt\, \csch t\,dt- \th t\, dt,\quad & \hbox{if $s=1$;} \\
 -\sh t\, dt (\cth t\,dt+1) (\cth t\,dt)^{s-2} \csch t\, dt, \phantom{\frac12} \quad & \hbox{if $s\ge 2$.}
 \end{array}
\right.
\end{align*}

\begin{thm} \label{thm-2ndVariant}
Set $\psi=\sh^{-1} 1=\log\nu$. For all $n\in\N_0$, $\bfs=(s_1,\dots,s_d)\in\N^d$ we have
\begin{align*}
\chi(\bfs;i\sh y)_n:=\sum_{n_1> \cdots>n_d>n} \frac{ b_{n_1}(i\sh y)}{(2n_1-1)^{s_1}\cdots (2n_d-1)^{s_d}}
= (-1)^d \frac{d}{d y} \int_0^y \tlk_{s_1}\circ\cdots \circ \tlk_{s_d}\circ b_{n}(i\sh t) \,dt .
\end{align*}
Here $y\in(-\psi,\psi)$ if $n_1=1$ and $y\in[-\psi,\psi]$ if $n_1\ge 2$.
\end{thm}

\begin{proof}
Applying the substitution $y\to iy$ in Thm.~\ref{thm-2ndVariantOld}
we can prove the theorem easily since $\csc it=-i\csch t$ and $\sech it=\sech t$.
\end{proof}

\section{Alternating Ap\'ery-type central binomial series}
We have studied Ap\'ery-type inverse binomial series in the previous sections. It is natural to see if the same idea works for
the binomial series, too. In \cite{XuZhao2022b} we successfully carried out this investigation and proved that results similar to those in Thm.~\ref{thm-gs-Akhilesh} and its odd-indexed versions, Thm.~\ref{thm-1stVariantOld} and Thm.~\ref{thm-2ndVariantOld}, still hold. We will generalize some of these to the alternating case in this section.

Put $a^+_n(x)=a_n(x):= {\binn} x^{2n}/4^n$ and $a^+_n=a_n:= {\binn}/4^n.$ Define
\begin{align*}
&f_{\pm 1}(t):= 1,\quad f_{\pm 2}(t):= \frac{t}{\sqrt{1\mp t^2}},\quad
f_{\pm 3}(t)=\frac1t, \quad f_{\pm 20}(t):= \frac{1}{t\sqrt{1\mp t^2}}, \quad f_5(t)=t.
\end{align*}
Recall that the 1-forms $\om_{\pm k}$ are defined by \eqref{rho:ChangVar0}--\eqref{rho:ChangVar4}.
Then for all subscript $k=1,2,3,4,5,20$ we have
$$
f_{\pm k}(t) \om_{\pm 1}=\om_{\pm k} \text{ where } \om_{\pm 20}=\om_0\pm \om_{\pm 2}.
$$
By \cite{XuZhao2022b}, for all $n\ge0$, $s\ge 1$, and $y\in(-\pi/2,\pi/2)$, we have
\begin{align*}
\sum_{m>n} \frac{a_m(\sin y)}{(2m)^s} =&\, \int_0^y (\cot t\,dt)^{s-1} (1-\csc t \,dt\circ\sec t)a_n(\sin t) \tan t\,dt, \\
\sum_{m>n} \frac{a_m(\sin y)}{(2m+1)^s} =&\,\csc y \int_0^y (\cot t\,dt)^{s-1} (1-dt\circ\csc t\sec t)a_n(\sin t) \sin t\tan t\,dt,\\
\sum_{m\ge n} \frac{a_m(\sin y)}{(2m+1)^s} =&\,\csc y\int_0^y (\cot t\,dt)^{s-1} (\csc t-dt\circ\sec t) a_n(\sin t) \tan t\,dt,\\
\sum_{m>n} \frac{a_m(\sin y)}{2m-1} =&\,\cos y\int_0^y a_n(\sin t) \tan t \sec t \, dt, \\
\sum_{m>n} \frac{a_m(\sin y)}{(2m-1)^s}
=&\, \sin y \int_0^y (\cot t\,dt)^{s-2} (\cot^2 t\,dt) a_n(\sin t) \tan t \sec t \, dt \quad (s\ge2).
\end{align*}
With substitution $\sin y=x$ we get
\begin{align}\label{equ-It1}
\sum_{m>n} \frac{a_m(x)}{(2m)^s} =&\, \int_0^x \om_0^{s-1} \Big(1-\om_{3}\circ \frac{1}{\sqrt{1-t^2}}\Big)a_n(t) \om_2,\\
\sum_{m\ge n} \frac{a_m(x)}{(2m+1)^s} =&\,\frac1x\int_0^x \om_0^{s-1} \Big(\frac1t-\om_1\circ\frac{1}{\sqrt{1-t^2}}\Big) a_n(t) \om_2, \nonumber \\
\sum_{m>n} \frac{a_m(x)}{2m-1} =&\,\sqrt{1-x^2}\int_0^x a_n(t) \frac{t\, dt}{(1-t^2)^{3/2}},\nonumber \\
\sum_{m>n} \frac{a_m(x)}{(2m-1)^s}
=&\, x \int_0^x \om_0^{s-2} \frac{\sqrt{1-t^2}\, dt}{t^2} a_n(t) \frac{t\, dt}{(1-t^2)^{3/2}} \quad (s\ge2).\nonumber
\end{align}
The above iteration formulas form the foundation of all the results in \cite{XuZhao2022b}.
Set $a^{-}_m(x)=(-1)^m a_m(x)$. Applying substitution $x\to ix$ in the above we obtain
\begin{align}
 \sum_{m>n} \frac{a^-_m(x)}{(2m)^s} =&\, \int_0^x \om_0^{s-1} \Big(\om_{-3}\circ \frac{1}{\sqrt{1+t^2}}-1\Big)a^-_n(t) \om_{-2},\label{an-ALT-it1}\\
\sum_{m\ge n} \frac{a^-_m(x)}{(2m+1)^s} =&\,\frac1x\int_0^x \om_0^{s-1} \Big(\frac1t+\om_{-1}\circ\frac{1}{\sqrt{1+t^2}}\Big) a^-_n(t) \om_{-2},\label{an-ALT-it3}\\
\sum_{m>n} \frac{a^-_m(x)}{2m-1} =&\,-\sqrt{1+x^2}\int_0^x a^-_n(t) \frac{t\, dt}{(1+t^2)^{3/2}}, \label{an-ALT-it4}\\
\sum_{m>n} \frac{a^-_m(x)}{(2m-1)^s}
=&\,x \int_0^x \om_0^{s-2} \frac{\sqrt{1+t^2}\, dt}{t^2} a^-_n(t) \frac{t\, dt}{(1+t^2)^{3/2}} \quad (s\ge2).\label{an-ALT-it5}
\end{align}
for all $s\ge 1$.
By repeatedly applying \eqref{an-ALT-it1}--\eqref{an-ALT-it5} we can find many results analogous to those in
section \ref{sec:invBinn}. See \cite{XuZhao2022b} for the approach we used to study the non-alternating version.

Fix a primitive 8th root of unity $\mu=(1+i)/\sqrt{2}$. Using the idea in the proof for the non-alternating case we can prove
the following theorem.

\begin{thm}\label{thm-AltAperyLevel8d}
Let $d\in \N$, $\bfs=(s_1,\dots,s_d)\in\N^d$ and $\eta=\pm1$.
Let $l_j(n)=2n$ or $l_j(n)=2n\pm 1$ for all $1\le j\le d$. Then
\begin{align}\label{equ-AltAperyLevel8d}
\sum_{n_1 \succ n_2 \succ\, \cdots \succ n_d\succ\,0}
 \frac{a_{n_1}\eta^{n_1}}{l_1(n_1)^{s_1}\cdots l_d(n_d)^{s_d}}\in\CMZV^8_{\le |\bfs|}\otimes\Q[i,\sqrt{2}],
\end{align}
where ``$\succ$'' can be either ``$\ge$'' or ``$>$'', provided the series is defined.
\end{thm}

\begin{proof}
If $\eta=1$ then the Ap\'ery-like sum \eqref{equ-AltAperyLevel8d} is in $\CMZV_{|\bfs|}^4$ by
\cite[Theorem 5.3]{XuZhao2022b}. If $\eta=-1$,
by the same proof for \cite[Theorem 5.3]{XuZhao2021b} we can show that this sum can be expressed
as an iterated integral involving only the following 1-forms:
\begin{alignat*}{3}
\text{First block:} &\qquad \csch t \,dt , \quad\cth t\,dt, \quad dt &&\quad (\text{cf. \cite[(5.16)]{XuZhao2021b}}),\\
\text{Mid blocks:} &\qquad \th t\,dt, \quad \sech t\csch t \,dt &&\quad (\text{cf. \cite[(5.21)]{XuZhao2021b}}),
\end{alignat*}
and no additional 1-forms can appear at the end block.
Under the change of variables $t\to \sh^{-1} y$ we have
\begin{align*}
& dt \to \om_{-1}, \quad
\cth t\,dt \to \om_0,\quad
\th t\,dt \to \om_{-2},\quad
\csch t \,dt \to \om_{-3}, \qquad
\sech t\csch t \to \om_{-20}.
\end{align*}
Under the change of variables $t\to i(1-t^2)/(1+t^2)$ we get
\begin{align} \label{equ:changeVar1}
\om_0=&\,\ta= \frac{dt}{t} \to \ty, \quad
    \om_{-1}=\frac{dt}{\sqrt{1+t^2}} \to \td_{i,-i}, \quad
    \om_{-2}=\frac{t\,dt}{1+t^2}\to -\tz, \\
\om_{-3}=&\, \frac{dt}{t\sqrt{1+t^2}}\to \td_{-1,1},  \quad
\om_{-20}=\frac{dt}{t(1+t^2)}\to \ty+\tz= -\ta-\tx_{-1}-\tx_{1}, 
\label{equ:changeVar3}
\end{align}
where $\ty=\tx_{-i}+\tx_{i}-\tx_{-1}-\tx_{1}$, $\tz=-\ta-\tx_{-i}-\tx_{i}$ and $\td_{\xi,\xi'}=\tx_\xi-\tx_{\xi'}$ for any two
roots of unity $\xi$ and $\xi'$.
We see that all the $\om$'s transform according to \eqref{equ:changeVar1}--\eqref{equ:changeVar3}.
Hence the sum \eqref{equ-AltAperyLevel8d} can be expressed by $\Q[i,\sqrt{2}]$-linear combination of
convergent iterated integrals of the form
\begin{align*}
\int_1^\mu \tx_{\ga_1}\cdots \tx_{\ga_k}, \quad \ga_j\in\{0, \mu^e:0\le e\le 7\}.
\end{align*}
Now we can apply Chen's theory of iterated integrals (see e.g., \cite[Lemma 2.1.2(ii)]{Zhao2016}) to get
\begin{align*}
\int_1^\mu \tx_{\ga_1}\cdots \tx_{\ga_k}=&\, \sum_{\ell=0}^k \left(\Reg\int_0^\mu \tx_{\ga_1}\cdots \tx_{\ga_\ell} \right)\left(\Reg\int_1^0 \tx_{\ga_{\ell+1}}\cdots \tx_{\ga_{k-1}}\tx_{\ga_k}\right)\\
=&\, (-1)^{k-\ell} \sum_{\ell=0}^k \left(\Reg\int_0^1 \tx_{\ga_1/\mu}\cdots \tx_{\ga_\ell/\mu}\right) \left(\Reg\int_0^1 \tx_{\ga_k} \tx_{\ga_{k-1}}\cdots \tx_{\ga_{\ell+1}}\right),
\end{align*}
where $\Reg$ means one should take regularized values of the possible divergent integrals (see \cite[13.3.1]{Zhao2016} for detailed treatment of this mechanism). These regularized values are all polynomials of $T$ with all the coefficients
given by CMZVs of level 8. Therefore the theorem follows by specializing at $T=0$. This completes the proof of the theorem.
\end{proof}

\begin{re} When computing concrete examples, sometimes one can avoid to use the full regularization mechanism. Instead, one can often use shuffle products and/or substitutions to combine all divergent integrals into convergent ones. See Example~\ref{ex-anOver2n-12m}
and Example~\ref{ex-bnOver2n-12m} for applications of these ideas.
\end{re}

We now present a few enlightening examples.
\begin{ex}\label{ex-anOver2n}
By \eqref{an-ALT-it1}, for all $s\ge 1$ we have
\begin{align*}
 \sum_{n>0}\frac{(-1)^n a_n(x)}{(2n)^s}
=&\, \int_0^x \om_0^{s-1} \Big(\om_{-3}\circ \frac{1}{\sqrt{1+t^2}}-1\Big) \om_{-2}\\
=&\, \int_0^x \om_0^{s-1} \Big( \frac{dt}{t\sqrt{1+t^2}}\circ \frac{t\, dt}{(1+t^2)^{3/2}}-\frac{t\, dt}{1+t^2}\Big)\\
=&\, \int_0^x \om_0^{s-1}  \Big(\frac{dt}{t\sqrt{1+t^2}} \cdot \left[\frac{-1}{\sqrt{1+u^2}}\right]_0^t - \frac{t\, dt}{1+t^2} \Big)\\
=&\, \int_0^x \om_0^{s-1} \big(\om_{-3}-\om_{-20}-\om_{-2}\big)= \int_0^1 \om_0^{s-1} \big(\om_{-3}-\om_0\big).
\end{align*}
Applying the change of variables $t\to i(1-t^2)/(1+t^2)$ and using \eqref{equ:changeVar1} we see that
\begin{align}\label{equ-anOver2ns}
 \sum_{n>0}\frac{(-1)^n a_n(x)}{(2n)^s}=&\, \int_{1}^{\gl(x)} \ty^{s-1}\tc,
\end{align}
where $\gl(x)=\sqrt{(1+ix)/(1-ix)}$ and $\tc=2\tx_{-1}-\tx_{-i}-\tx_{i}$.
We can take $x=\sqrt{j}/2$ ($1\le j\le 4$) in \eqref{equ-anOver2ns} to get
\begin{align*}
 \sum_{n>0}\frac{(-j)^n \binn}{16^n (2n)} =\left.\log\frac{t^2+1}{(t+1)^2}\right|_{1}^{\gl(\sqrt{j}/2)}
 =\log\Big(\frac4{j} (\sqrt{j+4}-2) \Big) 
.
\end{align*}
Setting $\ty_\mu=\tx_{-i/\mu}+\tx_{i/\mu}-\tx_{-1/\mu}-\tx_{1/\mu}$
and $\tc_{\mu}=2\tx_{-1/\mu}-\tx_{-i/\mu}-\tx_{i/\mu}$, and noticing that $\gl(1)=\mu$ we obtain
\begin{align}\label{equ-tc}
 \sum_{n>0}\frac{(-1)^n a_n}{(2n)^2} =&\, \int_{1}^{0} \ty\tc
 + \int_{0}^{\mu} \ty \int_{1}^{0}\tc
 + \int_{0}^{\mu} \ty\tc=\int_0^1 \tc\ty
-\int_{0}^1 \ty_\mu \int_0^1 \tc
 + \int_{0}^1 \ty_\mu \tc_{\mu}\\
=&\,\frac{\pi^2}8 - \frac12\Big(\log^2 2 - 2\log2 \log\nu + 2 \log^2\nu +4 \Li_2(\nu^{-1})\Big)
 \approx -0.1074917339.          \nonumber
\end{align}
\end{ex}

\begin{ex}\label{ex-anOver2n+1}
For all $s\ge 1$, using \eqref{an-ALT-it3} we have
\begin{align*}
 \sum_{n\ge 0}\frac{(-1)^n a_n(x)}{(2n+1)^s}
=&\, \frac1x \int_0^x \om_0^{s-1}\Big(\frac1t+\om_{-1}\circ \frac{1}{\sqrt{1+t^2}}\Big)\om_{-2} \\
=&\, \frac1x \int_0^x \om_0^{s-1}\Big(\frac{dt}{1+t^2}+\om_{-1}\circ\frac{t\, dt}{(1+t^2)^{3/2}}\Big) \\
=&\,  \frac1x \int_0^x \om_0^{s-1} \Big(\frac{dt}{1+t^2}-\om_{-1}\cdot \Big[\frac{1}{\sqrt{1+u^2}}\Big]_0^t\Big) \\
=&\,  \frac1x \int_0^x \om_0^{s-1} \om_{-1}= \frac1x  \int_{1}^{\gl(x)} \ty^{s-1}  \td_{i,-i} 
\end{align*}
by the change of variables $t\to i(1-t^2)/(1+t^2)$. Of course, when $x=\sqrt{j}/2$  ($1\le j\le 4$) and $s=1$
one can integrate without change of variables to get
\begin{equation*}
\sum_{n>0}\frac{(-j)^n \binn}{16^n (2n+1)} = \frac2{\sqrt{j}} \int_0^{\sqrt{j}/2} \om_{-1}
=\frac2{\sqrt{j}} \log(u+\sqrt{1+u^2})\big|_0^{\sqrt{j}/2}=\frac2{\sqrt{j}} \log\Big(\frac{\sqrt{j}+\sqrt{4+j}}2\Big).  
\end{equation*}
Similarly to Example~\ref{ex-anOver2n}, when $x=1$ and $s=2$ we can replace $\tc$ by $\td_{i,-i}$ in \eqref{equ-tc} to get
\begin{align*}
 \sum_{n>0}\frac{(-1)^n a_n}{(2n+1)^2} =&\, \int_0^1 \td_{i,-i}\ty
-\int_{0}^1 \ty_\mu \int_0^1 \td_{i,-i}
 + \int_{0}^1 \ty_\mu \td_{i/\mu,-i/\mu}\\
=&\,\frac{5\pi^2}{24}+\log2 \log\nu -\log^2\nu-2\Li_2(\nu^{-1})
 \approx 0.9552018.
\end{align*}

\end{ex}

\begin{ex}
For any $s\ge 1$, \eqref{an-ALT-it3} and \eqref{an-ALT-it1} and the computation in Example \ref{ex-anOver2n} yield that
\begin{align*}
 \sum_{n\ge m>0}\frac{(-1)^n a_n}{(2n+1)^s(2m)}
=&\, \int_0^1 \om_0^{s-1}\Big(\frac1t+\om_{-1}\circ \frac{1}{\sqrt{1+t^2}}\Big) \sum_{m>0}\frac{a^-_m(t)}{2m} \om_{-2}\\
=&\, \int_0^1 \om_0^{s-1}\Big(\frac1t+\om_{-1}\circ \frac{1}{\sqrt{1+t^2}}\Big)\om_{-2}\big(\om_{-3}-\om_0\big) \\
=&\, \int_0^1 \om_0^{s-1}\Big(\frac{dt}{1+t^2}+\om_{-1}\circ\frac{t\, dt}{(1+t^2)^{3/2}}\Big) \big(\om_{-3}-\om_0\big).
\end{align*}
Since
\begin{align*}
\int_{t_2}^{t_1} \frac{t\, dt}{(1+t^2)^{3/2}}=\frac{1}{\sqrt{1+t_2^2}}-\frac{1}{\sqrt{1+t_1^2}}
\end{align*}
we have
\begin{align*}
 \sum_{n\ge m>0}\frac{(-1)^n a_n}{(2n+1)^s(2m)}
=&\,\int_0^1 \om_0^{s-1}\om_{-1} \Big[\frac{1}{\sqrt{1+t^2}}\big(\om_{-3}-\om_0\big)\Big] \\
=&\,\int_0^1 \om_0^{s-1}\om_{-1} \big(\om_{-20}-\om_{-3}\big)=  \int_{1}^\mu \ty^{s-1} \td_{-i,i}\te
\end{align*}
by the change of variables $t\to i(1-t^2)/(1+t^2)$, where $\te=\tx_0+2\tx_{-1}$. If $s=1$ then we get
\begin{align*}
&\, \sum_{n\ge m>0}\frac{(-1)^n a_n}{(2n+1)(2m)}
= \int_{1}^\mu \td_{-i,i}\te
= \lim_{\eps\to0} \left( \int_{1}^\eps\td_{-i,i}\te+\int_\eps^\mu\td_{-i,i}\int_{1}^\eps\te+\int_\eps^\mu \td_{-i,i}\te \right)\\
=&\,\lim_{\eps\to0} \left( \int_{\eps}^1\te\td_{-i,i} -\int_\eps^\mu\td_{-i,i}\left(\int_\eps^\mu\te+\int_\mu^1\te\right)+\int_\eps^\mu \td_{-i,i}\te  \right)\\
=&\,\int_{0}^1\te\td_{-i,i} -\int_0^\mu\td_{-i,i}\int_\mu^1\te-\int_0^\mu  \te \td_{-i,i}
=\frac{5\pi^2}{48} - \log 2 \log\nu - \Li_2(\nu^{-1})\approx -0.0503718221
\end{align*}
by using Au's Mathematica package \cite{Au2020}. When $s=2$ we obtain
\begin{align*}
&\, \sum_{n\ge m>0}\frac{(-1)^n a_n}{(2n+1)^2(2m)}
=\lim_{\eps\to0} \left(  \int_{1}^\eps\ty\td_{-i,i}\te+\int_\eps^\mu\ty\int_1^\eps \td_{-i,i}\te +\int_\eps^\mu\ty\td_{-i,i}\int_{1}^\eps\te+\int_\eps^\mu \ty \td_{-i,i}\te\right)\\
=&\,\lim_{\eps\to0} \left( -\int_\eps^1\te\td_{-i,i}\ty+\int_\eps^\mu\ty\int_\eps^1 \te \td_{-i,i}-\int_\eps^\mu\ty\td_{-i,i}\left(\int_\eps^\mu\te+\int_\mu^1\te\right)+\int_\eps^\mu \ty \td_{-i,i}\te\right)\\
=&\,-\int_0^1\te\td_{-i,i}\ty+\int_0^\mu\ty\int_0^1 \te \td_{-i,i}-\int_0^\mu\ty\td_{-i,i}\int_\mu^1\te
    -\int_0^\mu \ty \te\td_{-i,i}-\int_0^\mu \te\ty \td_{-i,i}\\
=&\, \frac98\zeta(3)- \frac83\sqrt{2}L(3,\chi_8) - \frac{\pi^2}{48} \log2- \frac{\log^32}{12}
- \frac{3\log^22}{2} \log\nu + 2\log2 \log^2\nu-\frac23\log^3\nu \\
&\,   + (3\log2-2\log\nu) \Li_2(\nu^{-1}) + 4 \Li_3\Big(\frac1{\sqrt{2}}\Big) - 2  \Li_23(\nu^{-1})
\approx -0.02023197786,
\end{align*}
where $L(3,\chi_8)$ is the Dirichlet $L$-function with the primitive character $\chi_8$ modulo 8 satisfying $\chi_8(3)=\chi_8(5)=-1$ and $\chi_8(1)=\chi_8(7)=1$.
\end{ex}

\begin{ex}
With the same idea of iteration, \eqref{an-ALT-it4} and \eqref{an-ALT-it1} imply that
\begin{align*}
&\, \sum_{n>m>0}\frac{(-1)^n a_n}{(2n-1)(2m)}
= -\sqrt{2} \int_0^1 \frac{t\, dt}{(1+t^2)^{3/2}} \sum_{m>0}\frac{a^-_m(t)}{2m}\\
=&\, \sqrt{2}\int_0^1 \left[\frac{1}{\sqrt{1+t^2}}\right]_t^1 \Big(\om_{-3}-\om_0\Big) \\
=&\, \int_0^1 \Big(\om_{-3}-\om_0\Big)
 -\sqrt{2}\int_0^1 \frac{1}{\sqrt{1+t^2}} \Big(\om_{-3}-\om_0\Big) \\
=&\, \int_0^1 \Big(\om_{-3}-\om_0\Big)- \sqrt{2}\int_0^1 \om_{-20}-\om_{-3} \\
=&\, \nu\int_0^1 (\om_{-3}-\om_0)+\sqrt{2}\int_0^1\om_{-2}\\
=&\, \nu (\log 2 - \log (\sqrt2 + 1))+\frac{\sqrt{2}}2\log 2\\
=&\, \Big(1+\frac{3\sqrt{2}}2\Big) \log 2-\nu\log \nu\approx 0.035710328462762
\end{align*}
by Example~\ref{ex-anOver2n}. Note that the weight at the end drops by 1 as is the general case when $l_1(n)=2n-1$.
\end{ex}

\begin{ex}\label{ex-anOver2n-12m}
As a last example in this section, we consider a sum not covered by Thm.~\ref{thm-AltAperyLevel8d}
since the 1-form
$$
\om_6:=\frac{dt}{t\sqrt{1-t^4}}
$$
appears. Using \eqref{an-ALT-it4} and \eqref{equ-It1} we get
\begin{align*}
&\, \sum_{n>m>0}\frac{(-1)^{n+m} a_n}{(2n-1)(2m)}
= -\sqrt{2} \int_0^1 \frac{t\, dt}{(1+t^2)^{3/2}} \sum_{m>0}\frac{a_m(t)}{2m}\\
=&\, \sqrt{2}\int_0^1 \left[\frac{1}{\sqrt{1+t^2}}\right]_t^1 \Big(1-\om_{3}\circ \frac{1}{\sqrt{1-t^2}}\Big) \om_2\\
=&\, \int_0^1 \Big(\om_{3}+\om_2-\om_{20}\Big) -\sqrt{2}\int_0^1 \frac{1}{\sqrt{1+t^2}} \Big(\om_{3}+\om_2-\om_{20}\Big) \\
=&\, \int_0^1 \Big(2\tx_{-1}-\tx_{i}-\tx_{-i}\Big)+\sqrt{2}\int_0^1 \frac{1}{\sqrt{1+t^2}} \Big(\ta-\om_{3}\Big)
 \quad\text{(by $t\to \frac{1-t^2}{1+t^2}$ in the first integral)}\\
=&\, \log 2 +\sqrt{2}\int_0^1\Big(\om_{-3}-\om_6\Big) .
\end{align*}
To compute the last integral we use the regularization trick as follows:
\begin{align*}
\int_0^1\Big(\om_{-3}-\om_6\Big)
&\, =\lim_{\eps\to 0} \Big(\int_\eps^1 \om_{-3}- \int_\eps^1 \om_6 \Big)
 =\lim_{\eps\to 0} \Big(\int_\eps^1 \om_{-3}- \frac12 \int_{\eps^2}^1 \om_3\Big)
\end{align*}
by the change of variables $t\to \sqrt{t}$ in the second integral. Hence
by the change of variables $t\to \frac{i(1-t^2)}{1+t^2}$ in the first integral and
 $t\to \frac{1-t^2}{1+t^2}$ in the second integral we see that
\begin{align*}
\int_0^1\Big(\om_{-3}-\om_6\Big)
&\, =\lim_{\eps\to 0} \int_{\gl(\eps)}^{\gl(1)} \td_{-1,1} - \frac12 \int_{\gl(\eps^2)}^0 \td_{-1,1}.
\end{align*}
Since $\gl(-i)=\mu$ we have
\begin{align*}
\int_0^1\Big(\om_{-3}-\frac{dt}{t\sqrt{1-t^4}}\Big)
&\, =\lim_{\eps\to 0}\left( \left. \log\frac{t-1}{t+1}\right|_{\gl(-i\eps)}^{\gl(-i)}
 - \frac12 \left.\log\frac{t-1}{t+1}\right|_{\gl(\eps^2)}^0\right)\\
&\, = \log\frac{\mu-1}{\mu+1}+\lim_{\eps\to 0} \left(\log\frac{\sqrt{1+i\eps}+\sqrt{1-i\eps}}{\sqrt{1+i\eps}-\sqrt{1-i\eps}}
+ \frac12 \log\frac{\sqrt{1+\eps^2}-\sqrt{1-\eps^2}}{\sqrt{1+\eps^2}+\sqrt{1-\eps^2}}\right)\\
&\, = \log\frac{\mu-1}{\mu+1}+ \log (-\sqrt{2} i)=\log \sqrt{2} +\log\frac{i(1-\mu)}{\mu+1}=\log \sqrt{2}-\log \nu.
\end{align*}
Thus
\begin{align*}
 \sum_{n>m>0}\frac{(-1)^{n+m} a_n}{(2n-1)(2m)}
 =&\, \Big(1+\frac{\sqrt{2}}2\Big)\log 2-\sqrt{2} \log\nu \approx -0.063174227986
\end{align*}
which is clearly a value in $\CMZV_1^8\otimes\Q[i,\sqrt{2}]$.
\end{ex}

\section{Alternating Ap\'ery-type inverse central binomial series with summation indices of mixed parities}
We now return to the inverse binomial series and consider their alternating versions in this section.
Keeping the same notation as before, we write $b^+_n(x):=b_n(x)$ and
\begin{align*}
 b^-_n(x)=b_n(ix)=(-1)^n 4^n{\binn}^{-1}x^{2n}\quad \text{and}\quad b^-_n= b^-_n(1)=(-1)^n 4^n{\binn}^{-1}\quad \forall n\ge 0.
\end{align*}
Combining \cite[(4.3)--(4.8)]{XuZhao2022a} with Theorems \ref{thm-gsbar1-ItInt}, \ref{thm-1stVariant} and \ref{thm-2ndVariant} yield that
\begin{align}
\sum_{n_1>n} \frac{ b^\pm_{n_1}(x)}{2n_1} = &\,\pm f_{\pm 2}(x)\int_0^x b^\pm_{n}(t) \om_{\pm 1}, \label{bn-ALT-it1} \\
 \sum_{n_1>n} \frac{ b^\pm_{n_1}(x)}{(2n_1)^s}=&\, \pm f_1(x) \int_0^x \om_0^{s-2} \om_{\pm 1}\, b^\pm_{n}(t)\om_{\pm 1}\quad\forall s\ge 2,\label{bn-ALT-it2} \\
 \sum_{n_1\ge n} \frac{ b^\pm_{n_1}(x)}{2n_1+1}=&\,f_{\pm20}(x)\int_0^x b^\pm_{n}(t) \om_{\pm1} , \label{bn-ALT-it3} \\
 \sum_{n_1\ge n} \frac{ b^\pm_{n_1}(x)}{(2n_1+1)^s}=&\,f_3(x)\int_0^x \om_0^{s-2} \om_{\pm3}\, b^\pm_{n}(t) \om_{\pm1} \quad\forall s\ge 2, \label{bn-ALT-it4} \\
 \sum_{n_1>n} \frac{ b^\pm_{n_1}(x)}{2n_1-1} = &\, \pm f_5(x) \int_0^x \om_{\pm3}\, b^\pm_{n}(t) \om_{\pm1}\pm f_{\pm2}(x) \int_0^x b^\pm_{n}(t) \om_{\pm1}, \label{bn-ALT-it5} \\
 \sum_{n_1>n} \frac{ b^\pm_{n_1}(x)}{(2n_1-1)^{s}} = &\,\pm f_5(x) \int_0^x (\om_0+1) \om_0^{s-2} \om_{\pm3}\, b^\pm_{n}(t) \om_{\pm1}\quad\forall s\ge 2. \label{bn-ALT-it6}
\end{align}

By applying \eqref{bn-ALT-it1}-\eqref{bn-ALT-it6} iteratively it is possible to express every alternating Ap\'ery-type inverse binomial series with summation indices of mixed parities by an iterated integral involving only 1-forms $\om_{\pm j}$ ($0\le j\le 6$) where
$\om_{\pm 5}=\frac{t\, dt}{\sqrt{1\mp t^2}}$ and $\om_{\pm 6}=\frac{dt}{t\sqrt{1-t^4}}$. Unfortunately, $\om_{\pm 3}$, $\om_{\pm 5}$
and $\om_{\pm 6}$ transform badly under the change of variables \eqref{equ:changeLevel8}.
Using the idea in the proof for the non-alternating case we can slightly extend Thm.~\ref{thm-AltAperyLevel8} to the following more general form.

\begin{thm}\label{thm-AltAperyLevel8b}
Let $d\in \N$, $\bfs=(s_1,\dots,s_d)\in\N^d$ and $\bfeta=(\eta_1,\dots,\eta_d)\in \{\pm1\}^d$.
Let $l_j(n)=2n$ or $l_j(n)=2n+1$ for all $1\le j\le d$ so that $s_j=1$ and $\eta_j=\sgn(j-1.5)$ if $l_j(n)=2n+1$.
If $(s_1,\eta_1)\ne (1,1)$ then
\begin{align*}
 \sum_{n_1>\dots>n_d>0}
 \frac{b_{n_1}\eta_1^{n_1}\cdots \eta_d^{n_d}}{l_1(n_1)^{s_1}\cdots l_d(n_d)^{s_d}}\in\CMZV^8_{|\bfs|}\otimes\Q[i,\sqrt{2}].
\end{align*}
\end{thm}

\begin{proof} We only need to consider the appearance of $l_j(n)=2n+1$ raised to the first power on the denominator.
If it only appear at the beginning then we need to have $\eta_1=-1$ to guarantee convergence, in which case
the theorem is clear since \eqref{bn-ALT-it3} only involves the 1-form $\om_{\pm 1}$.
If $l_j(n)=2n+1$ appears for some $j\ge 2$ then \eqref{bn-ALT-it3} may produce $\om_{\pm 20}=\om_0\pm \om_{\pm 2}$
after combining $f_{\pm 20}$ with the 1-form $\om_{\pm 1}$ produced by the proceeding iteration since $\eta_j=1$
guarantees the same sign pattern. The rest of the proof follows from the same reasoning as used in that of Thm.~\ref{thm-AltAperyLevel8} and thus is left to the interested reader.
\end{proof}

\begin{re} If $l_j(n)=2n+1$ appears for some $j\ge 2$ but $\eta_j=-1$ then the 1-form $\om_6$ is produced. We
do not know how to handle this in general with our approach yet.
\end{re}

The follow result as well as its proof is analogous to Thm.~\ref{thm-AltAperyLevel8d}
\begin{thm}\label{thm-AltAperyLevel8c}
Let $d\in \N$, $\bfs=(s_1,\dots,s_d)\in\N^d$ and $\eta=\pm1$.
Let $l_j(n)=2n$ or $l_j(n)=2n\pm 1$ for all $1\le j\le d$. If $(s_1,\eta_1)\ne (1,1)$ then
\begin{align}\label{equ-AltAperyLevel8c}
\sum_{n_1 \succ n_2 \succ\, \cdots \succ n_d\succ\,0}
 \frac{b_{n_1}\eta^{n_1}}{l_1(n_1)^{s_1}\cdots l_d(n_d)^{s_d}}\in\CMZV^8_{\le |\bfs|}\otimes\Q[i,\sqrt{2}],
\end{align}
where ``$\succ$'' can be either ``$\ge$'' or ``$>$'', provided the series is defined.
\end{thm}

\begin{proof} The proof is almost exactly the same as that of Thm.~\ref{thm-AltAperyLevel8d} so we omit the details here.
The key idea is to follow the proof of \cite[Theorem 4.3]{XuZhao2022a} step by step but use the substitution $t\to i(1-t^2)/(1+t^2)$
(instead of $t\to (1-t^2)/(1+t^2)$ in the original proof) at the end.
\end{proof}

We now consider several alternating Ap\'ery-type inverse binomial series with summation indices having mixed parity.
\begin{ex}
Using \eqref{bn-ALT-it3} and \eqref{bn-ALT-it1} successively we get
\begin{align*}
&\, \sum_{n\ge m>0}\frac{(-1)^n b_{n}}{(2n+1) (2m)}=\sum_{n\ge m>0}\frac{b^-_{n}(1)}{(2n+1)(2m)}
= \frac1{\sqrt{2}} \int_0^1 \sum_{m>0} \frac{b^{-}_m(t)}{(2m)} \om_{-1}
= \frac{-1}{\sqrt{2}} \int_0^1 \om_{-2} \om_{-1}\\
=&\, \frac14\int_0^1 \left(\tx_i+\tx_{-i}-\frac12\sum_{l=1}^4 \tx_{\mu_l}\right)
 \sum_{k=1}^4 (\mu_k-\mu_k^3)\tx_{\mu_k} \\
= &\, \frac18\sum_{k=1}^4\sum_{\eps=\pm i} (\mu_k-\mu_k^3) \int_0^1 \tx_\eps \tx_{\mu_k}
-\frac14 \sum_{k,l=1}^4 (\mu_k-\mu_k^3) \int_0^1 \tx_{\mu_l}\tx_{\mu_k} \\
= &\,\frac14\sum_{k,l=1}^4 (\mu_k-\mu_k^3) \Li_{1,1}(\mu_l^{-1},\mu_l/\mu_k)
 -\frac18\sum_{k=1}^4\sum_{\eps=\pm i}(\mu_k-\mu_k^3) \Li_{1,1}(\eps^{-1},\eps/\mu_k)\\
=&\, \frac{3\sqrt{2}}{32}\Big(8\log^2\nu+16\Li_2(\sqrt{2}-1)-8\log 2\log\nu-\pi^2\Big)
\approx -0.14078648719
\end{align*}
by Au's Mathematica package \cite{Au2020}. We have also verified this by numerically
computing the series directly and by
numerically evaluating the integral with Mathematica:
\begin{equation*}
\frac1{\sqrt{2}}\int_0^1 \om_{-2} \om_{-1}=\frac1{\sqrt{2}}\int_0^1 \frac{t\ln(t+\sqrt{1+t^2})}{1+t^2}\, dt
\approx -0.14078648719.
\end{equation*}
\end{ex}

\begin{ex}
In depth 3, using \eqref{bn-ALT-it1}, \eqref{bn-ALT-it3}, and \eqref{bn-ALT-it2} successively yields that
\begin{align*}
&\, \sum_{n>m\ge k>0}\frac{(-1)^n b_{n}}{(2n)(2m+1) (2k)^2}=\sum_{n>m\ge k>0}\frac{b^-_{n}(1)}{(2n)(2m+1) (2k)^2}\\
=&\, \frac{-1}{\sqrt{2}} \int_0^1 \sum_{m\ge k>0} \frac{b^{-}_m(t)}{(2m+1) (2k)^2} \om_{-1}
=\frac{-1}{\sqrt{2}} \int_0^1 \om_{-20} \sum_{k>0} \frac{b^-_k(t)}{(2k)^2} \om_{-1}
= \frac1{\sqrt{2}} \int_0^1 \om_{-20} \om_{-1}^3\\
=&\, \frac{1}{8}\int_0^1 \left(\ta+\tx_i+\tx_{-i}\right)
 \left( \sum_{k=1}^4 (\mu_k-\mu_k^3)\tx_{\mu_k} \right)^3\\
= &\,\frac{1}{32} \sum_{j,k,l=1}^4 (\mu_j-\mu_j^3)(\mu_k-\mu_k^3)(\mu_l-\mu_l^3)
 \Li_{2,1,1}(\mu_j^{-1}, \mu_j/\mu_k,\mu_k/\mu_l)\\
&\,+\frac{1}{32}\sum_{j,k,l=1}^4\sum_{\eps=\pm i} (\mu_j-\mu_j^3)(\mu_k-\mu_k^3)(\mu_l-\mu_l^3)
 \Li_{1,1,1,1}(\eps^{-1},\eps/\mu_j, \mu_j/\mu_k,\mu_k/\mu_l)\\
=&\,\frac{7\sqrt{2}}{64}G^2+\frac{3739}{8192}\pi^4+\text{other terms in }\CMZV_4^8 \approx 0.0202649114985
\end{align*}
by Au's Mathematica package \cite{Au2020}. We have also checked this by numerically
computing the series directly and by
numerically evaluating the integral with Mathematica:
\begin{equation*}
 \frac1{\sqrt{2}} \int_0^1 \om_{-20} \om_{-1}^3=\frac1{6\sqrt{2}}\int_0^1 \frac{\log^3(t+\sqrt{1+t^2})}{t(1+t^2)}\, dt
 \approx 0.0202649114985.
\end{equation*}
\end{ex}

\begin{ex}
Similarly, using \eqref{bn-ALT-it3}, \eqref{bn-ALT-it2},
and \eqref{bn-ALT-it1} successively we can get
\begin{align*}
&\, \sum_{n\ge m>k>0}\frac{(-1)^{n+k} b_{n}}{(2n+1)(2m)^2 (2k)}=\sum_{n\ge m>k>0}\frac{b^-_{n}(1)(-1)^{k}}{(2n+1)(2m)^2 (2k)}\\
=&\, \frac1{\sqrt{2}} \int_0^1 \sum_{m>k>0} \frac{b^{-}_m(t)(-1)^{k}}{(2m)^2 (2k)} \om_{-1}
=\frac{-1}{\sqrt{2}} \int_0^1 \om_{-1}^2 \sum_{k>0} \frac{b_k(t)}{(2k)} \om_{-1}
= \frac{-1}{\sqrt{2}} \int_0^1 \om_{-1}^2 \om_4 \om_{1}\\
=&\, \frac{-1}{64}\int_0^1
 \sum_{h=1}^4 (\mu_h-\mu_h^3)\tx_{\mu_h} \sum_{j=1}^4 (\mu_j-\mu_j^3)\tx_{\mu_j}
 \sum_{l=1}^4 \mu_l^2 \tx_{\mu_l} \sum_{k=1}^4 (\mu_k+\mu_k^3)\tx_{\mu_k} \\
= &\,\frac{-1}{64} \sum_{h,j,k,l=1}^4 \mu_l^2(\mu_h-\mu_h^3)(\mu_j-\mu_j^3)(\mu_k+\mu_k^3)
 \Li_{1,1,1,1}(\mu_h^{-1},\mu_h/\mu_j, \mu_j/\mu_l,\mu_l/\mu_k) \\
= &\, \frac{\sqrt{2}}{128}\Big( 3\pi^3 \log(1 + \sqrt2)-\pi^3 \log2 -
 128 \Im\Li_3\Big(\frac{1+i}2\Big) \log(1 + \sqrt2)\\
&\, +\pi \big(4 \log^2 2 \log(1 + \sqrt2)-
 8\log2 \log^2(1 + \sqrt2) +3 \zeta(3)\big)\Big)
\approx -0.00777369894.
\end{align*}
We have also checked this by numerically
computing the series directly and by
numerically evaluating the integral with Mathematica:
\begin{align*}
\frac{-1}{\sqrt{2}} \int_0^1 \om_{-1}^2 \om_4 \om_{1}=&\,\frac{-1}{2\sqrt{2}}
\int_0^1 \frac{t(\log\nu-\log(t+\sqrt{1+t^2}))^2\log(t+\sqrt{1+t^2})}{\sqrt{1-t^4}}\, dt \\
 \approx&\, -0.00777369894.
\end{align*}
\end{ex}

\begin{ex}
This example converges very slowly. Using \eqref{bn-ALT-it3}, \eqref{bn-ALT-it2},
and \eqref{bn-ALT-it1} successively we see that
\begin{align*}
&\, \sum_{n\ge m>k>0}\frac{(-1)^n b_{n}}{(2n+1)(2m)^2 (2k)}=\sum_{n\ge m>k>0}\frac{b^-_{n}(1)}{(2n+1)(2m)^2 (2k)}\\
=&\, \frac1{\sqrt{2}} \int_0^1 \sum_{m>k>0} \frac{b^{-}_m(t)}{(2m)^2 (2k)} \om_{-1}
=\frac{-1}{\sqrt{2}} \int_0^1 \om_{-1}^2 \sum_{k>0} \frac{b^-_k(t)}{(2k)} \om_{-1}
= \frac1{\sqrt{2}} \int_0^1 \om_{-1}^2 \om_{-2} \om_{-1}\\
=&\, \frac{1}{32}\int_0^1
 \sum_{h=1}^4 (\mu_h-\mu_h^3)\tx_{\mu_h} \sum_{j=1}^4 (\mu_j-\mu_j^3)\tx_{\mu_j}
 \left(\frac12\sum_{l=1}^4 \tx_{\mu_l}-\tx_i-\tx_{-i}\right)
 \sum_{k=1}^4 (\mu_k-\mu_k^3)\tx_{\mu_k} \\
= &\,\frac{1}{64} \sum_{h,j,k,l=1}^4 (\mu_h-\mu_h^3)(\mu_j-\mu_j^3)(\mu_k-\mu_k^3) \int_0^1 \tx_{\mu_h} \tx_{\mu_j} \tx_{\mu_l}\tx_{\mu_k}\\
&\, -\frac{1}{32}\sum_{h,j,k=1}^4\sum_{\eps=\pm i} (\mu_h-\mu_h^3)(\mu_j-\mu_j^3)(\mu_k-\mu_k^3) \int_0^1 \tx_{\mu_h} \tx_{\mu_j} \tx_\eps \tx_{\mu_k}\\
= &\,\frac{1}{64} \sum_{h,j,k,l=1}^4 (\mu_h-\mu_h^3)(\mu_j-\mu_j^3)(\mu_k-\mu_k^3)
 \Li_{1,1,1,1}(\mu_h^{-1},\mu_h/\mu_j, \mu_j/\mu_l,\mu_l/\mu_k)\\
&\, -\frac{1}{32}\sum_{h,j,k=1}^4\sum_{\eps=\pm i} (\mu_h-\mu_h^3)(\mu_j-\mu_j^3)(\mu_k-\mu_k^3)
 \Li_{1,1,1,1}(\mu_h^{-1},\mu_h/\mu_j, \mu_j/\eps,\eps/\mu_k)\\
\approx &\, 0.00585622967.
\end{align*}
We have also checked this by numerically
computing the series directly and by
numerically evaluating the integral with Mathematica:
\begin{align*}
\frac1{\sqrt{2}} \int_0^1 \om_{-1}^2 \om_{-2} \om_{-1}=&\,\frac{1}{2\sqrt{2}}
\int_0^1 \frac{t(\log\nu-\log(t+\sqrt{1+t^2}))^2\log(t+\sqrt{1+t^2})}{1+t^2}\, dt \\
 \approx&\, 0.00585622967.
\end{align*}
\end{ex}

\begin{ex}\label{ex-bnOver2n-12m}
Using \eqref{bn-ALT-it5}, and \eqref{bn-ALT-it1} successively we find that
\begin{align*}
 \sum_{n> m>0}\frac{(-1)^{n} b_{n}}{(2n-1)(2m)}
= &\, -\int_0^1 \om_{-3}\sum_{m>0} \frac{b^-_m(t)}{2m} \om_{-1}-\frac1{\sqrt{2}} \int_0^1 \sum_{m>0} \frac{b^-_m(t)}{2m} \om_{-1}\\
=&\,\int_0^1 \om_{-3} \om_{-2} \om_{-1}+\frac{1}{\sqrt{2}} \int_0^1 \om_{-2} \om_{-1}.
\end{align*}
Under the change of variables $t\to i(1-t^2)/(1+t^2)$ for the first integral (see \eqref{equ:changeVar1})
and \eqref{equ:changeLevel8} for the second we get
\begin{align*}
&\, \sum_{n> m>0}\frac{(-1)^{n} b_{n}}{(2n-1)(2m)}
= \int_1^\mu \td_{-1,1}\tz \td_{-i,i}+\frac{1}{4} \int_0^1
 \Big(\frac12\sum_{j=1}^4 \tx_{\mu_j}-\tx_i-\tx_{-i} \Big)\sum_{k=1}^4 (\mu_k-\mu_k^3)\tx_{\mu_k}.
\end{align*}
For some very small $\eps>0$, by Chen's theory of iterated integrals (see, e.g., \cite[Lemma 2.1.2(ii)]{Zhao2016}) we may compute
the first integral as
\begin{align*}
\int_1^\mu \td_{-1,1}\tz \td_{-i,i}=&\, \int_1^\eps \td_{-1,1}\tz \td_{-i,i}
+\int_\eps^\mu \td_{-1,1}\int_1^\eps \tz \td_{-i,i}
+\int_\eps^\mu \td_{-1,1}\tz \int_1^\eps \td_{-i,i}
+\int_\eps^\mu \td_{-1,1}\tz \td_{-i,i}\\
=&\, -\int_\eps^1 \td_{-i,i}\tz \td_{-1,1}
+\int_\eps^\mu \td_{-1,1}\int_\eps^1 \td_{-i,i}\tz
-\int_\eps^\mu \td_{-1,1}\tz \int_\eps^1 \td_{-i,i}
+\int_\eps^\mu \td_{-1,1}\tz \td_{-i,i}.
\end{align*}
Note that by the shuffle relation
\begin{align*}
&\,\int_\eps^\mu \td_{-1,1}\int_\eps^1 \td_{-i,i}\tz
-\int_\eps^\mu \td_{-1,1}\tz \int_\eps^1 \td_{-i,i} \\
=&\,\int_\eps^\mu \td_{-1,1}\left(\int_\eps^1 \td_{-i,i}\int_\eps^1\tz -\int_\eps^1\tz\td_{-i,i}\right)
-\left(\int_\eps^\mu \td_{-1,1}\int_\eps^\mu \tz -\int_\eps^\mu \tz\td_{-1,1} \right)\int_\eps^1 \td_{-i,i}.
 \end{align*}
Thus taking $\eps\to 0$ we have
\begin{align*}
&\, \sum_{n>m>0}\frac{(-1)^{n} b_{n}}{(2n-1)(2m)}
= -\int_0^1 \td_{-i,i}\tz \td_{-1,1}-\int_0^\mu\td_{-1,1} \int_0^1\tz\td_{-i,i}
+\int_0^\mu \td_{-1,1}\int_0^1\td_{-i,i}\int_\mu^1\tz \\
&\, +\int_0^\mu \tz\td_{-1,1} \int_0^1 \td_{-i,i}+\int_0^\mu \td_{-1,1}\tz \td_{-i,i}
+\frac{1}{4} \int_0^1 \Big(\frac12\sum_{j=1}^4 \tx_{\mu_j}-\tx_i-\tx_{-i} \Big) \sum_{k=1}^4 (\mu_k-\mu_k^3)\tx_{\mu_k}\\
= &\,-\int_0^1 \td_{-i,i}\tz \td_{-1,1}-\int_0^1\td_{-1/\mu,1/\mu} \int_0^1\tz\td_{-i,i}
+\int_0^1 \td_{-1/\mu,1/\mu}\int_0^1\td_{-i,i}\int_\mu^1\tz \\
&\, -\int_0^1 (\ta+\tx_{-i/\mu}+\tx_{i/\mu})\td_{-1/\mu,1/\mu} \int_0^1 \td_{-i,i}
-\int_0^1 \td_{-1/\mu,1/\mu}(\ta+\tx_{-i/\mu}+\tx_{i/\mu}) \td_{-i/\mu,i/\mu} \\
&\, +\frac{1}{4} \Big(\frac12\sum_{j=1}^4 \tx_{\mu_j}-\tx_i-\tx_{-i} \Big)\sum_{k=1}^4 (\mu_k-\mu_k^3)\tx_{\mu_k}\\
=&\, \frac{\pi^2}{32}\Big(3 \sqrt2 + 4\log2 - 6 \log\nu\Big) +
 \frac{1}{6} \Big(16\sqrt2 L(3,\chi_8) + \log^3\nu \Big) + \frac{\log 2 }{4} (3 \sqrt2 - 2 \log\nu)\log\nu\\
 &\,- \frac{3 \sqrt2}{4} \Big( \log^2\nu+\Li_2(\nu^{-1})\Big)-4\Li_3(\nu^{-1})- \log\nu \Li_2(\nu^{-1})-\frac{3}{8} \zeta(3) \approx 0.20569096448
\end{align*}
by using Au's Mathematica package \cite{Au2020}.
\end{ex}

In view of the theorems and the examples obtained so far, we would like to conclude the paper with the following question.

\medskip
\noindent
\textbf{Question.} Let $d\in \N$, $\bfs=(s_1,\dots,s_d)\in\N^d$ and $\bfeta=(\eta_1,\dots,\eta_d)\in \{\pm1\}^d$.
Let $l_j(n)=2n$ or $l_j(n)=2n\pm 1$ for all $1\le j\le d$. Is it true that
\begin{align*}
 \sum_{n_1 \succ n_2 \succ\, \cdots \succ n_d\succ\,0}
 \frac{a_{n_1}\eta_1^{n_1}\cdots \eta_d^{n_d}}{l_1(n_1)^{s_1}\cdots l_d(n_d)^{s_d}}\in\CMZV^8_{\le|\bfs|}\otimes\Q[i,\sqrt{2}],\\ \sum_{n_1 \succ n_2 \succ\, \cdots \succ n_d\succ\,0}
 \frac{b_{n_1}\eta_1^{n_1}\cdots \eta_d^{n_d}}{l_1(n_1)^{s_1}\cdots l_d(n_d)^{s_d}}\in\CMZV^8_{\le |\bfs|+1}\otimes\Q[i,\sqrt{2}]
\end{align*}
if the sums converge? Here ``$\succ$'' can be either ``$\ge$'' or ``$>$''.

\bigskip

\noindent{\bf Acknowledgement.} Ce Xu is supported by the National Natural Science Foundation of China [Grant No. 12101008], the Natural Science Foundation of Anhui Province [Grant No. 2108085QA01] and the University Natural Science Research Project of Anhui Province [Grant No. KJ2020A0057]. Jianqiang Zhao is supported by the Jacobs Prize from The Bishop's School.

\end{document}